\RequirePackage{rotating}
\documentclass[a4paper]{amsart}
\usepackage{amsmath, amsfonts,amsthm,amssymb,amscd, verbatim,graphicx,color,multirow,booktabs, caption, tabularx,,bm,chngcntr,adjustbox,longtable,mathtools,microtype,booktabs, stackengine,scalerel}

\usepackage[top = 78pt, bottom = 72pt, inner=108pt, outer=108pt]{geometry}% the inner and outer shall be reduced to 72pt in the final version
\usepackage[dvipsnames, table]{xcolor}
\usepackage{hyperref}
\hypersetup{colorlinks=true,linkcolor=Maroon, citecolor=OliveGreen}

\stackMath

\usepackage[capitalise,noabbrev]{cleveref}
\usepackage[shortlabels]{enumitem}

\newcommand{\primozcomment}[1]{\textcolor{magenta}{{\textbf{Primoz:} #1}}}

\newcommand{\Sym}{\mathrm{Sym}}
\newcommand{\Fun}{\mathrm{Fun}}
\newcommand{\Aut}{\mathrm{Aut}}
\newcommand{\Alt}{\mathrm{Alt}}
\newcommand{\Out}{\mathrm{Out}}
\renewcommand{\wr}{\mathop{\mathrm{wr}}}

\newcommand{\cE}{\mathcal{E}}

\newcommand{\f}{\mathbf{f}}
\newcommand{\F}{\mathbf{F}}
\newcommand{\FF}{\mathbb{F}}

\newcommand{\m}{\mathbf{m}}
\newcommand{\K}{\mathbf{K}}

\newcommand{\C}{\mathcal{C}}
\newcommand{\D}{\mathcal{D}}

\newcommand{\NN}{\mathbb{N}}
\newcommand{\mix}{\mathop{\mathrm{mix}}}
\newcommand{\qrk}{\mathbf{q}}

\newcommand{\A}{\mathsf{A}}
\newcommand{\act}{\,\rotatebox[origin=c]{-180}{$\circlearrowright$}\,}
\newcommand{\PGammaL}{\rm P\Gamma L}
\newcommand{\AGammaL}{\rm A\Gamma L}
\newcommand{\PGL}{\rm PGL}

\newcommand{\PSL}{\rm PSL}
\newcommand{\Dih}{\rm Dih}
\newcommand{\ASL}{\rm ASL}
\newcommand{\AGL}{\rm AGL}
\newcommand{\one}{\mathbf{1}}

\title{Separating subsets from their images}

\author[M.~Barbieri]{Marco Barbieri}
\address{Faculty of Mathematics and Physics, University of Ljubljana, Jadranska ulica 21, 1000 Ljubljana, Slovenia.} 
\email{marco.barbieri@fmf.uni-lj.si}

\author[M.~Lekše ]{Maruša Lekše}
\address{Institute of Mathematics, Physics, and Mechanics, Jadranska ulica 19, 1000 Ljubljana, Slovenia. Also affiliated with: Faculty of Mathematics and Physics, University of Ljubljana, Jadranska ulica 21, 1000 Ljubljana, Slovenia.} 
\email{marusa.lekse@imfm.si}

\author[P.~Potočnik]{Primož Potočnik}
\address{Faculty of Mathematics and Physics, University of Ljubljana, Jadranska ulica 21, 1000 Ljubljana, Slovenia. Also affiliated with: Institute of Mathematics, Physics, and Mechanics, Jadranska ulica 19, 1000 Ljubljana, Slovenia.}
\email{primoz.potocnik@fmf.uni-lj.si} 

\author[K.~Rekv\'enyi]{Kamilla Rekv\'enyi} 
\address{Department of Mathematics, University of Manchester,  M13 9PL Manchester, UK. Also affiliated with: Heilbronn Institute for Mathematical Research, BS8 1UG Bristol, UK.}
\email{kamilla.rekvenyi@manchester.ac.uk}


\newtheorem{theorem}{Theorem}[]
\newtheorem{lemma}[theorem]{Lemma}

\newtheorem{thmx}{Theorem}[]

\newtheorem{corx}[thmx]{Corollary}

\theoremstyle{definition}
\newtheorem{de}[theorem]{Definition}
\newtheorem{example}[theorem]{Example}
\newtheorem{remark}[theorem]{Remark}
\newtheorem{question}[thmx]{Problem}
\newtheorem{vprasanje}[theorem]{Question}
\newtheorem{construction}[theorem]{Construction}

\thanks{The research presented in this paper was initiated during a research visit by MB and KR to IMFM, funded by Slovenian Research and Innovation Agency, research programme number P1-0294. They gratefully acknowledge the support and hospitality provided. MB is supported by the Slovenian Research Agency programme P1-0222 and grant J1-50001, and he is a member of the Italian GNSAGA INdAM research group. ML and PP are supported by the Slovenian Research and Innovation Agency, programme number P1-0294.}

\begin{document}
\maketitle
\begin{abstract}
Let $G$ be a transitive permutation group acting on $\Omega$. In this paper, we introduce and study the parameter $\m(G)$, which denotes the size of the smallest set of points $A$ such that, for every  permutation $g\in G$, $A \cap A^g$ is nonempty. In particular, we focus on deriving general bounds for arbitrary transitive groups, and on the asymptotic behaviour of certain families of primitive groups. We also provide a classification of transitive groups with $\m(G)$ largest possible, namely with $\m(G)=\lceil (|\Omega|+1) / 2 \rceil$. 
\end{abstract}

\tableofcontents
	
\section{Introduction}\label{sec:intro}
Let us begin by presenting a graph-theoretical problem which spurred our interest in the topic of the paper and led us to define and investigate the notions presented here.

Suppose we are given a simple finite graph $\Gamma$ on $n$ vertices and we want to know whether its complement contains an isomorphic copy of $\Gamma$ as a subgraph, or equivalently, whether two copies of $\Gamma$ can be packed into the complete graph $\K_n$. Clearly, if $\Gamma$ contains more than $2^{-1} {n\choose 2}$ edges, then such a packing is impossible. On the other hand, if $\Gamma$ is very sparse (for example, if it has maximum valence $1$ and at least $3$ vertices), then the packing will very likely exist. It is thus natural to ask what is the largest integer $m$ such that every graph on $n$ vertices and at most $m$ edges can be packed twice into $\K_n$. Since such a packing is clearly impossible for the complete bipartite graph $\K_{1,n-1}$, we see that $m\le n-2$. Surprisingly, as was elegantly proved in \cite[Theorem 1]{BurnsSchuster1977}, this bound is sharp, in the sense that every graph on $n$ vertices and at most $n-2$ edges lies in its complement as a subgraph.

%Let $\Gamma$ be a graph with vertex set $[h]=\{1,2,\dots,h\}$. A classical question in graph theory is whether the complement of $\Gamma$ isomorphic to $\Gamma$ or not. To tackle this question we can focus in determining the minimum number of edges for which there exists a graph $\Gamma$ that does not enjoy this property. For instance, if the number of edges of $\Gamma$ is $h-1$, then the star $\K_{1,h-1}$ is not contained in $\K_h - \K_{1,h-1}$, as all the vertices of the complement have valency $h-2$. Actually, \cite[Theorem 1]{BurnsSchuster1977} states that, for every graph $\Gamma$ containing $h-2$ edges or less, $\K_h - \Gamma$ contains an isomorphic copy of $\Gamma$. This result is the bedrock of \emph{packing theory} (see \cite{GorlichZak2010,GorlikPilsniakWozniakZiolo2007} for recent developments).

The above problem for graphs can actually be translated into the language of permutation groups. Suppose $\Gamma$ is a graph with vertex set $[n]:=\{1,2,\ldots,n\}$ and edge set $E \subseteq {[n] \choose 2}$. An isomorphism from $\Gamma$ to a graph  contained in the complement of $\Gamma$, is a permutation of $[n]$ which, when viewed as a permutation on the unordered pairs ${[n] \choose 2}$, maps $E$ to a set disjoint from $E$. In short, two copies of $\Gamma$ can be packed into $\K_n$ if and only if the group $\Sym(n)$ in its action on ${[n] \choose 2}$ contains an element $g$ such that $E\cap E^g = \emptyset$. We say that such a set $E$ is self-separable, and if there is no such element $g$ then we say it is non-self-separable (see Definition~\ref{def:main}). The graph-theoretical theorem mentioned above is equivalent to the assertion that the size of the smallest non-self-separable set for the action of $\Sym(n) $ on ${[n] \choose 2}$ is $n-1$.

%This restatement now ask for treatment in a more general setting of arbitrary group actions. 
In this paper we show that viewing the original packing problem in this general group-theoretical context reveals connections to several classical problems and theorems in combinatorics, finite geometry and group theory, as well as
opens a plethora of new and interesting questions. We will address some of these and suggest further directions to explore.% provide some further questions for possible new directions.

%The previous request can be rephrased by asking whether there exists a permutation $g\in \Sym(n)$ such that $A \cap A^g$ is empty. The purpose of this paper is to study a permutation-group-theoretical parameter that generalizes this classical packing problem, and to highlight how this notion is closely related to several areas of combinatorics, finite geometry and group theory.

%\subsection{Definition}
\begin{de}
\label{def:main}
Let $G$ be a transitive permutation group on a finite set $\Omega$ with $|\Omega| \ge 2$. 
% of size $n \ge 2$. 
A subset $A\subseteq \Omega$ is said to be \emph{self-separable} for $G$ provided that there exists $g\in G$ such that $A\cap A^g = \emptyset$, and is {\em non-self-separable} otherwise. Further, let 
\[ \m(G) :=  \min \left\{ |A| \in \mathbb{N} \mid A\subseteq \Omega \hbox{ is not self-separable for }G \right\}\]
be the size of a smallest non-self-separable set for $G$.
\end{de}
%It is clear that every singleton is self-separable.
%Clearly, every subset of a self-separable set is self-separable, or equivalently, every overset of a non-self-separable set is itself non-self-separable.
%This raises a natural question of determining the size
%\[ \m(G \act \Omega) :=  \min \left\{ |A| \in \mathbb{N} \mid A\subseteq \Omega \hbox{ is not self-separable for }G \right\}\]
Observe that the parameter $\m(G)$ is also equal to the largest integer such that all sets of size less than that integer are self-separable, that is,
\[ \m(G) = \max \left\{ t \in \mathbb{N} \mid \forall A \subseteq \Omega : |A|< t  \Rightarrow A \hbox{ is self-separable} \right\}.\]

The main theme of this paper is to address the following:
\begin{question}\label{prob:0}
Given a transitive permutation group $G$, determine the parameter $\m(G)$.
\end{question}
The above question can be approached from several angles, ranging from algorithmic considerations to exact theoretical results. We mainly focus on theoretical ways, but briefly consider the computational approach too.

The computational task of determining $\m(G)$ for a given permutation group $G$ appears to be highly challenging. The only algorithm we devised during the preparation of this paper involves an exhaustive search (see~\cite{ourAlgorithm} and~\cref{sec:upperBound}) using a theoretical lower bound we establish in \cref{thm:neumann}. Our algorithm systematically tests all subsets $A$ of the domain whose sizes exceed the theoretical lower bound in \cref{thm:neumann}, and stops when we find the minimal example $A$ that is non-self-separable for $G$. When $G$ is known to be $k$-homogeneous, we may limit the search space to subset containing a fixed $k$-subset. No further improvements of this naïve approach are known to us. Therefore we ask the following:

\begin{vprasanje}
    Does there exist a significantly more efficient algorithm for computing the parameter $\m(G)$ given a permutation group $G$ (perhaps, subexponential in the degree)?
\end{vprasanje}
%\kalacomment{It is out of character to write comments not in pink for me but I don't want to confuse them with Primož's pink comments.}
%\primozcomment{What are you talking about. My text is in magenta and not pink!}
%\kalacomment{ I think this is an interesting question here for sure, but it creates more focus on computation than what we are giving it later on, are we happy with that?}
%\primozcomment{I thought it fits well, as well as points out that we were not really able to come up with a fast algorithm. But feel free to erase it, as long as you do it so that the text still flows nicely.}
In this paper, however, we will be especially interested in asymptotic (lower and upper) bounds on $\m(G)$ in terms of the degree of $G$. ~\cref{thm:neumann} tells us that, for every group $G$ of order $n$, the parameter $\m(G)$ always satisfies $\sqrt{n} \le \m(G) \le \lceil (n+1)/2 \rceil$.
Moreover, infinite families of transitive groups exist for which $\m(G)$ is asymptotically square root of the degree (for example,
the family of symmetric groups acting on unordered pairs, which is our motivating example) as well as infinite families where $\m(G)$ is asymptotically linear in the degree of the group (for example, the symmetric groups in their natural action). 
%What we find interesting, though, is that there seems to be a dichotomy between these square root and linear asymptotical behaviours.

To understand asymptotic behaviours better, we need the following definitions. For an infinite family $\C$ of finite permutation groups and for an integer $n$, let $\C_n$ be the subfamily of permutation groups in $\C$ that have degree $n$, and let
\[ X _ {\C}:= \left\{ n\in \mathbb{N} \mid \C_n \ne \emptyset \right\} .\]
The asymptotic behaviour of the parameter $\m(G)$ for $G \in \C$ will be analysed in terms of the functions:
\[ \f_\C: X _ {\C} \to \mathbb{N}, \quad n \mapsto 
\min \left\{\m (G) \mid G \in \C_n \right\} ,\]
and
\[ \F_\C: X _ {\C} \to \mathbb{N}, \quad n \mapsto 
\max \left\{\m (G) \mid G \in \C_n \right\} .\]

\begin{remark}
\label{rem:sub}
We can immediately observe that, for every two families $\D\subseteq \C$ and for every positive integer $n\in X_\C$,
\[\f_\C(n) \le \f_\D(n) \le \F_\D(n) \le \F_\C(n) .\]
\end{remark}

To understand the behaviour of these functions on large class of transitive permutation groups, we will consider the following problem.
\begin{vprasanje}\label{prob:1}
    Given an infinite class $\C$ of transitive permutation groups, what is the asymptotic behaviour of $\f_\C(n)$ and of $\F_\C(n)$? In particular, do there exist two positive constants $r \in \left(\frac{1}{2} , 1\right)$ and $c$, and an infinite family $\C$
    such that 
    \[ \liminf_n \frac{\f_\C(n)}{n^{r}} = \limsup_n \frac{\F_\C(n)}{n^{r}} = c ?\]
    That is, does there exist a family $\C$ such that the expected asymptotic behaviour of $\m(G)$, for a generic $G\in \C$, is neither square root nor linear?
\end{vprasanje}

To give a taste of the statements that we will be proving, let us list some of them.
In the case of the class $\C_{{\rm tr}}$ of all transitive permutation groups, we prove that $\liminf \f_{\C_{{\rm tr}}}(n)$ grows asymptotically as $\sqrt{n}$ (see Theorem~\ref{thm:upperBound}), while
$\limsup \F_{\C_{{\rm tr}}}(n)$ is asymptotically $\frac{n}{2}$ (see Theorem~\ref{thm:neumann}), or to be precise:
\[
\liminf_n \frac{\f_{\C_{{\rm tr}}}(n)}{\sqrt{n}} = 1 \quad \hbox{ and } \quad
\limsup_n \frac{\F_{\C_{{\rm tr}}}(n)}{n} = \frac{1}{2}.
\]
Further, as is stated in Theorem~\ref{thm:regular}, for the class $\C_{{\rm reg}}$ of regular permutation groups, the asymptotic bound for
$\F_{\C_{{\rm reg}}}$ can be significantly improved to
\[
\limsup_n \frac{\F_{\C_{{\rm reg}}}(n)}{\sqrt{n}} \le \frac{4}{\sqrt{3}}.
\]
Particular attention is also given to the class $\C_{{\rm pr}}$ of primitive permutation groups, specifically to the following problem.

\begin{question}\label{prob:primitive}
    Determine what is the largest subclass $\C \subseteq \C_{{\rm pr}}$ where a bound of the form
    \[
    \limsup_n \frac{\F_{\C}(n)}{\sqrt{n}} \le C
    \]
    holds
    for some constant $C$.
\end{question}

To facilitate smooth reading, the next section provides an extended abstract of the results proved in this paper. The proofs can be found in later sections.

\subsection{Acknowledgements} We thank Urban Jezernik for drawing our attention to the graph-theoretic problem that motivated this work. We are also grateful to Pablo Spiga for several helpful conversations, particularly regarding Neumann’s Separation Lemma, and for identifying an error in an earlier version of the proof of \cref{lemma:complementB}.

\section{Results}
\subsection{General bounds}
We start our investigation by determining some general bounds on $\m(G)$.

\begin{thmx}\label{thm:neumann}
    Let $G$ be a transitive permutation group on $\Omega$ with $|\Omega|=n$, and let $\alpha \in \Omega$ be a point. Then
    \begin{equation}\label{eq:lowerBound}
    \m(G) \ge \frac{1}{2|G_\alpha|} + \sqrt{n - \frac{1}{|G_\alpha|} + \frac{1}{4|G_\alpha|^2}},
    \end{equation}
    with equality if and only if $G$ is a regular group of collineations
    of a projective plane. In particular, if $\mathcal{C_{\mathrm{tr}}}$ is the class of all transitive permutation groups, then
    \[ \liminf_n \frac{\f_\mathcal{C_{\mathrm{tr}}}(n)}{\sqrt{n}} = 1 .\]
\end{thmx}
The proof of \cref{thm:neumann} (given in \cref{sec:lowerBound}) relies on the Neumann Separation Lemma (see~\cite{ BirchBurnsMacdonaldNeumann1976, Neumann1975}), and the characterisation of the equality depends on considerations from incidence geometry.

%We are not aware of any example in which $G$ is not regular and $\m(G)$ is the ceiling.

%Inequality~\eqref{eq:lowerBound}.
%the smallest integer satisfying the inequality of \cref{thm:neumann}.
%It would be interesting to either build such an example or to improve the bound under the assumption that the stabiliser of a point is nontrivial.

We now focus on an upper bound for $\m(G)$. Clearly, for every transitive permutation group $G$ of degree $n$, two subsets containing strictly more than $\lfloor  n / 2 \rfloor$ points cannot have trivial intersection. Hence,
\begin{equation}\label{eq:upperBound}
    \m(G) \le \left\lceil \frac{n + 1}{2} \right\rceil.
\end{equation} 
%The following example shows that the bound is sharp.

%\begin{example}\label{rem:Sym}
%Let $G$ be either the symmetric group $\Sym(n)$ or the alternating group $\mathrm{Alt}(n)$ endowed with their natural actions on $[n]$. Since $G$ is $\lfloor n / 2 \rfloor$-transitive, every subset containing half of the points or less can be mapped to its complement by a permutation in $G$. Hence, every subset $A$ containing $\lfloor  n / 2 \rfloor$ points is self-separable for $G$. Therefore,
%\[ \m(\Sym(n)) = \m(\mathrm{Alt}(n)) =  \left \lfloor \frac{n}{2} \right \rfloor + 1 = \left\lceil \frac{n + 1}{2} \right\rceil .\]
%\end{example}

Classifying the groups that attain the upper bound in Inequality~\eqref{eq:upperBound} is a nontrivial problem and is interesting in its own right.
%posing challenges both computationally and theoretically. 
For example, a very natural subproblem that needs to be addressed is:
%for even degree, the problem assumes the following very natural form: 
``{\em Determine all primitive permutation groups $G$ of even degree such that every subset of the domain of size half of the degree can be mapped to its complement by an element of $G$}.'' 
Clearly, the alternating and symmetric groups have this property,
but it is somewhat surprising that so do some other permutation groups.
%In \cref{sec:upperBound} we completely classify all such groups.
%This characterisation is summarised in the theorem below (see Lemmas~\ref{lemma:complementB}, \ref{lemma:caseA} and \ref{lemma:caseB} for details).

\begin{thmx}
\label{theorem:comput}
Let $G$ be a transitive permutation group of degree $n$. If $G$ attains the upper bound in Inequality~\eqref{eq:upperBound}, then $G$ satisfies one of the following conditions.
\begin{enumerate}[$(a)$]
    \item $G$ is primitive, and either $n \le 24$ or $\Alt(\Omega)\le G \le \Sym(\Omega)$. (The precise list of such groups can be found in Lemma~\ref{lemma:complementB}.)
    \item $G$ is imprimitive with a system of
    imprimitivity consisting of $n/2$ blocks of size $2$, and either $n\le 24$ or the action on the block system is isomorphic to either $\Alt(n/2)$ or $\Sym(n/2)$.  (The precise list of such groups can be found in Lemma~\ref{lemma:caseA}.)
    \item $G$ is imprimitive with a system of
    imprimitivity consisting of two blocks of size $n/2$, and either $n\le 24$ or the action induced on a block is isomorphic to either $\Alt(n/2)$ or $\Sym(n/2)$.
    (The precise list of such groups can be found in Lemma~\ref{lemma:caseB}.)
 \end{enumerate}
\end{thmx}

We stress that not all groups with the conditions above reach the putative upper bound. A complete characterisation is provided in \cref{sec:upperBound} (once again, see \cref{lemma:complementB,lemma:caseA,lemma:caseB}).

Since \cref{theorem:comput} yields infinite families of permutation groups attaining the bound in Inequality~\eqref{eq:upperBound}, we have the following consequence:
%The infinitude of permutation groups listed in the above theorems
%meeting the upper bound \eqref{eq:upperBound} 
%yields the following consequence:
\begin{thmx}\label{thm:upperBound}
    Let $\mathcal{C}_{\mathrm{tr}}$ be the class of all transitive permutation groups. Then
    \[ \lim_n \frac{\F_\mathcal{C_{\mathrm{tr}}}(n)}{n} = \frac{1}{2} .\]
\end{thmx}

%We provide this classification in \cref{sec:upperBound}. 
%Broadly speaking, unless the degree of $G$ is less than $24$, $G$ is either $\Alt(n)$ or $\Sym(n)$ in their natural actions on $n$ points, or it is isomorphic to a subgroup of $C_2 \wr \Sym(m)$ or $\Sym(m) \wr C_2$ endowed with the imprimitive action on $2m$ points (where the permutation group that $G$ induces on the system of blocks or on a single block, respectively, is either alternating or symmetric). For a precise statement, including the complete list of small-degree groups meeting the bound, we refer to Lemmas~\ref{lemma:complementB} and~\ref{lemma:complementC}.

\subsection{Permutation groups with bounded stabilisers}
In the previous section, we determined the asymptotic behaviour of $\f_{\C_{\mathrm{tr}}}$ and  $\F_{\C_{\mathrm{tr}}}$ for the class of all transitive permutation groups. Note, however, that the gap between the square root asymptotics of $\liminf\f_{\C_{\mathrm{tr}}}(n)$ and the linear growth of $\F_{\C_{\mathrm{tr}}}(n)$ leaves many questions about the parameter $\m(G)$ for a generic transitive permutation group $G$ open. In what follows, we will try to identify subclasses $\C \subseteq \C_{\mathrm{tr}}$ for which we are able to exert a more effective control. That is, families for which the gap between $f_{\C_{\mathrm{tr}}}$ and  $F_{\C_{\mathrm{tr}}}$ is smaller.
%As a result, our understanding of the typical behaviour of the functions $\f_\C$ and $\F_\C$ is primarily informed by certain restricted families $\C$ for which we are able to exert effective control. 
%In the remainder of this introduction, we outline our results concerning such families. 
In fact, we will observe that for many important families of permutation group $\C \subseteq \C_{\mathrm{tr}}$, $\F_\C(n)$ is asymptotically proportional to $\sqrt{n}$ (as well as $\f_\C(n)$, by consequence). The class of regular permutation groups provides such a family.

\begin{thmx}\label{thm:regular}
    Let $G$ be a regular permutation group of degree $n$. Then
    \[  \m(G) \le \frac{4}{\sqrt3} \sqrt{n} .\]
    In particular, if $\C_{\mathrm{reg}}$ is the family of all regular permutation groups, then
    \[1 = \liminf_n \frac{\f_{\mathcal{C}_{\mathrm{reg}}}(n)}{\sqrt{n}}  \le  \limsup_n \frac{\F_{\mathcal{C}_{{\mathrm{reg}}}}(n)}{\sqrt{n}} \le \frac{4}{\sqrt3} .\]
\end{thmx}
We will prove this result in \cref{sec:reg}.
This result follows from \cref{rem:regular}, which asserts that for a subset $A$ not being self-separable for a regular group $G$ is equivalent to $A$ being a difference basis for $G$:

\begin{de}\label{de:differenceBasis}
    A subset $A$ of a group $G$ such that $AA^{-1} = G$ is a \emph{difference basis for $G$}
\end{de}

Let us remark that asymptotic results for the minimal size of a difference basis for a generic group $G$ is still an open problem in additive combinatorics.

If every nontrivial element of $G$ can be written uniquely as $ab^{-1}$ for some $a,b \in A$, then the difference basis $A$ is called a \emph{difference set}, which is a widely used tool in finite geometries. For instance, if $q$ is the order of a classical projective plane, a paper 
\cite{Singer1938} by Singer
proves the existence of a difference set of size $q+1$ for the cyclic group of order $q^2+q+1$. Observe that these permutation groups are precisely those that meet the lower bound in \cref{thm:neumann}.

Let $G$ be an abstract group and suppose that $G$ induces multiple faithful actions on some finite sets $\Omega_i$. To distinguish them, we will use the symbol $G \act \Omega_i$ to denote the permutation groups induced by the action of $G$ on the domain $\Omega_i$. For instance, if $G=\mathrm{PSL}_d(q)$ and $V^k$ denotes the set of $k$-dimensional subspaces of the natural module, then the permutation group induced by the natural action of $G$ on $V^k$ can be denoted by $\mathrm{PSL}_d(q) \act V^k$. Similarly, if $H$ is a core-free subgroup of $G$, we denote the permutation group that $G$ induces on the right coset space $G/H$ by $G \act G/H$.

Let $H\le K$ be two core-free subgroups of $G$. In \cref{lemma:nestedstabilsers}, we establish that $\m(G\act G/K) \le \m(G \act G/H)$. By combining this fact with \cref{thm:regular}, we obtained the following corollaries of theorems of Cameron, Praeger, Saxl and Seitz \cite{CameronPraegerSaxlSeitz1983} and of Gardiner, Trofimov and Weiss \cite{Gardiner1976, Trofimov1990, Trofimov1991, Weiss1979, Weiss1993}.
\begin{corx}\label{cor:subdegree}
    For every positive integer $d$, there exists a constant $\mathbf{c}(d)$ such that, for every primitive permutation group $G$ of degree $n$ and minimal nontrivial subdegree at most $d$, the following inequality holds
    \[ \m(G) \le \mathbf{c}(d)\sqrt{n} .\]
    In particular, if $\C$ is the family of primitive permutation groups of minimal nontrivial subdegree at most $d$, then
    \[ \limsup_n \frac{\F_\C(n)}{\sqrt n} \le \mathbf{c}(d) .\]
\end{corx}
\begin{corx}\label{cor:graphRestrictive}
    For every positive integer $d$, there exists a constant $\mathbf{c}(d)$ such that, for every $2$-arc-transitive group of automorphisms $G$ of a connected graph with $n$ vertices and of valency at most $d$, the following inequality holds
    \[ \m(G) \le \mathbf{c}(d)\sqrt{n} .\]
    In particular, if $\C$ is the family of $2$-arc-transitive group of automorphisms of connected graphs of valency at most $d$, then
    \[ \limsup_n \frac{\F_\C(n)}{\sqrt n} \le \mathbf{c}(d) .\]
\end{corx}
The proofs of these results are in \cref{sec:reg}.
\subsection{Reduction to primitive permutation groups} A standard approach in permutation group theory is to investigate whether and to what extent the problem can be reduced to the setting of primitive groups.

Recall that every imprimitive permutation group $G$ can be reduced to an iterated wreath product of primitive permutation groups as follows. Let $\Sigma$ be a system of imprimitivity for $G$ and let $B$ be a block in $\Sigma$. As usual, we denote the permutation group that the setwise stabiliser $G_B$ induces on $B$ by $G_B^B$, and  the permutation group that $G$ induces on $\Sigma$ by $G^\Sigma$. Then $G$ embeds into the imprimitive wreath product $G_B^B \wr G^\Sigma$. Moreover, if $\Sigma$ is a coarsest system of imprimitivity, then $G^\Sigma$ is primitive, meanwhile if $\Sigma$ is a finest system of imprimitivity, then $G_B^B$ is primitive.

(The following theorem is proved in \cref{sec:extension}.)

\begin{thmx}\label{thm:boundsTransitive}
    Let $G$ be a group, and suppose that $G$ embeds into the imprimitive wreath product $H \wr K$. Then
    \[ \m(K) + \m(H) - 1 \le \m(G) \le \m(K) \m(H) \,.\]
\end{thmx}

The following example shows that the upper bound is asymptotically sharp. Let $K$ and $H$ be isomorphic to the cyclic group of degree $q^2+q+1$ with regular action, and let $G=K \wr H$. Note that the degree of $G$ is $(q^2+q+1)^2$. For large $q$, using \cref{thm:neumann} to find the lower bound for $\m(G)$ and the upper bounds for $\m(K)$ and $\m(H)$,
\[ q^2 + o(q^2) \le \m(G) \le \m(K) \m(H) \le (q + o(q))^2 = q^2 + o(q^2). \]
This proves the asymptotic sharpness of the upper bound of \cref{thm:boundsTransitive}. 

Furthermore, if $K$ and $H$ are regular groups, using the above argument and \cref{thm:regular}, we can show that $\m(G) \le \frac{8}{3}\sqrt{n}$ -- matching the optimal lower bound up to a multiplicative constant.

On the other hand, we are unable to prove the sharpness of the lower bound. 
We show, however, that there exists a family of permutation groups for which the lower bound is asymptotically tight up to a constant. The following result, of independent interest, will be essential in finding such a family. 
\begin{thmx}\label{thm:largeBlocks}\label{cor:blocks}
    If $G$ is the wreath product $\Sym(a) \wr \Sym(b)$ endowed with the imprimitive action on $ab$ points, then
    \[\m(G) = \begin{cases}
    a+b-2 \quad \hbox{if $a=3$ and $b$ is odd, or $b=3$ and $a$ is odd} \\
    a+b-1 \quad \hbox{otherwise} \,.
    \end{cases}\]
    In particular, if $G$ is an imprimitive permutation group with a block system of cardinality $b$ whose blocks contain $a$ points, then
    \[ \m(G) \le a + b - 1 \,.\]
\end{thmx}
The proof of this result is given in \cref{sec:ingredient}.
Let $a$ and $b$ be two even positive integers, and let $K=\Sym(a)$ and $H= \Sym(b)$. By choosing $G = K \wr H$, \cref{cor:blocks} implies that
\[ \m(G) = a + b - 1 = 2\m(K) + 2\m(H) - 5 .\]
On the other hand, \cref{thm:boundsTransitive} yields
\[ \m(G) \ge \m(K) + \m(H) - 1 .\]
If we allow $a$ and $b$ to grow as two arbitrary unbounded functions, we find that, asymptotically, $\m(G)$ is twice the lower bound of \cref{thm:boundsTransitive} -- proving the asymptotic sharpness up to the multiplicative constant $2$. This leads to the following questions.

\begin{vprasanje}
    Is the lower bound stated in \cref{thm:boundsTransitive} asymptotically sharp? If it is not, is $2(\m(K) + \m(H))$ an asymptotically sharp lower bound? Furthermore, by excluding an appropriate infinite family of transitive groups, is it possible to obtain a larger lower bound?
\end{vprasanje}

\subsection{Primitive permutation groups} Recall that, for $\Alt(n)$ and $\Sym(n)$ in their action on $n$ points, $\m(\Alt(n)) = \m(\Sym(n)) = \lceil (n+1)/2 \rceil$. At the moment, they constitute the single infinite family of primitive permutation groups where we can prove that a linear asymptotic behaviour is exhibited. We suspect that, possibly apart from few well-understood infinite families, the size of the smallest non-self-separable subset is asymptotically proportional to the square root of the degree of the group. This intuition is captured by the following question.

\begin{vprasanje}\label{prob:2}\label{quest:thisOne}
    Does there exist a constant $C$ such that, for the infinite class $\C_{{\rm pr}}^\ast$ of all primitive permutation groups apart from alternating and symmetric groups in their natural action,
    \[ \limsup_n \frac{\F_{\C_{{\rm pr}}^\ast}(n)}{\sqrt{n}} = C ?\]
\end{vprasanje}

Natural candidates that might exhibit a different behaviour are primitive permutation groups of product action type whose socle is a direct product of alternating groups in their natural action (see~\cref{cor:PA} and~\cref{lemma:complementB}). If the answer to \cref{prob:2} has a negative answer, then we still propose \cref{prob:primitive}, already stated in \cref{sec:intro}, as a possible direction of investigation.

In view of this and the already established square-root lower bound (see~\cref{thm:neumann}), we focus on obtaining upper bounds for $\m(G)$, and thus avoiding the computation of $\f_{{\C_{{\rm pr}}^\ast}}(n)$. Let us consider the following example.

\begin{example}
    Let $ \mathrm{PSL}_d(q) \unlhd G \le \mathrm{P}\Gamma \mathrm{L}_d(q)$ act on the projective space $\mathrm{PG}(d-1,q)$. To find an upper bound on $\m(G)$, one needs to find a subset $A$ of the projective space such that every image of $A$ under a collineation always intersects $A$. An obvious candidate is a projective subspace of dimension at least %exceeding \primozcomment{shouldn't we say `at least' rather than `exceeding'}
    half of that of the ambient $\mathrm{PG}(d-1,q)$. In this case, we can check that if $d-1$ is even, we can take a subspace of dimension $(d-1)/2$, and we find that, for large $n$,
    \[\m(G) \le \sqrt{n}+o(\sqrt{n}).\]
    Otherwise, if $d-1$ is odd, we take a subspace of dimension $d/2$, and hence we find
    \[ \m(G) \le n^{\frac{d}{2(d-1)}} + o \left( n^{\frac{d}{2(d-1)}} \right) .\]
    This construction is not meaningful for $d=2$, since in that case we cannot take any projective subspace apart from the line itself. In fact, we know very little of the behaviour of $\m(G)$ where $G \le \Aut(PG(1,q))$ and $q$ is large. Furthermore, even for $d \ge 4$, the exponent $d/(2(d-1))$ is always larger than $1/2$ (though it converges to $1/2$ as $d$ goes to infinity),  suggesting that these groups might also provide a negative answer to \cref{prob:2}. However, we believe that subsets distinct from subspaces might provide an example of smaller non-self-separable set. It would be fascinating if some subgeometry $A$ of $\mathrm{PG}(d-1,q)$ were not self-separable, but all the obvious candidates (such as ovoids or Galois subgeometries) seem to be self-separable or too large.
\end{example}

Similar reasoning can produce a large array of upper bounds for primitive groups of almost simple type in their standard action. As in \cite{LiebeckShalev2003}, the term \emph{standard actions}, refers to permutation groups that are either alternating, symmetric or classical, and, in the former two cases, the domain consists of $k$-subsets, while in the latter the groups are endowed with a \emph{subspace action}. (For the list of subspace actions, we refer to \cite[Chapter~4]{BurnessGiudici2016}.)

{\small
\begin{table}[th]
	\rowcolors{2}{white}{OliveGreen!25}
	\centering
	\begin{tabularx}{\textwidth}{l l X c}
		\toprule
		& $\mathrm{soc}(G)$ & Domain $\Omega$ & $\m(G \act \Omega)$ \\
		
		\midrule
        $(a)$ & $\mathrm{Alt}(m)$ & $k$-subsets & $\displaystyle
        %{m-\left\lfloor \frac{k}{2}\right\rfloor \choose k-\left\lfloor \frac{k}{2}\right\rfloor}
        \lesssim \frac{k!^{\frac{1}{k} \left\lceil \frac{k}{2}\right\rceil}}{\left\lceil \frac{k}{2}\right\rceil !} n^{\frac{1}{k} \left\lceil \frac{k}{2}\right\rceil}  $ \\
		
		$(b)$ & $\mathrm{PSL}_d(q)$ & $k$-subspaces & $\displaystyle
        %\left[ d-\lfloor k/2\rfloor \atop k-\lfloor k/2\rfloor\right]_q
        \lesssim n^{\frac{1}{k} \left\lceil \frac{k}{2}\right\rceil}$  \\

        $(c)$ & $\mathrm{PSL}_d(q)$ & \shortstack{pairs $(X,Y)$ of subspaces in\\direct sum, with $\dim(X)=k$} & $\displaystyle
        %q^{k(d-k)}\left[ d-\lfloor k/2\rfloor \atop k-\lfloor k/2\rfloor\right]_q
        \lesssim n^{\frac{1}{k} \left\lceil \frac{k}{2}\right\rceil}$ \\
        
        $(d)$ & $\mathrm{PSL}_d(q)$ & \shortstack{pairs $(X,Y)$ of subspaces with\\$X\le Y$, $\dim(X) =k$, $\dim(Y) =d-k$} & $\displaystyle
        \lesssim n^{\frac{k \left\lceil \frac{d - 3k}{2} \right\rceil + (d-k) \left\lceil \frac{k}{2} \right\rceil}{k(2d-3k)}}$  \\

        $(e)$ & $\mathrm{PSU}_d(q)$ & totally isotropic $k$-subspaces & $(\clubsuit)$ \\
        
        $(f)$ & $\mathrm{PSU}_d(q)$ & nondegenerate $k$-subspaces & $\displaystyle \lesssim n^{\frac{1}{k} \left\lceil \frac{k}{2}\right\rceil}$ \\

        $(g)$ & $\mathrm{PSp}_{2d}(q)$ & totally isotropic $k$-subspaces & $(\clubsuit)$ \\
        
        $(h)$ & $\mathrm{PSp}_{2d}(q)$ & nondegenerate $2k$-subspaces & $\displaystyle \lesssim n^{\frac{1}{k} \left\lceil \frac{k}{2}\right\rceil}$ \\
        
        $(i)$ & $\mathrm{P}\Omega_{2d}^+(q)$ & totally singular $k$-subspaces & $(\clubsuit)$ \\

        $(j)$ & $\mathrm{P}\Omega_{2d}^-(q)$ & totally singular $k$-subspaces & $(\clubsuit)$ \\

        $(k)$ & $\Omega_{2d+1}(q)$ & totally singular $k$-subspaces, $q$ odd & $(\clubsuit)$ \\
        
        $(l)$ & $\Omega_{2d}^\epsilon(q)$ & nonsingular $1$-subspaces, $q$ even & $\le \lfloor\frac{n}{2}\rfloor$ \\

        $(m)$ & $\mathrm{P}\Omega_{2d}^+(q)$ & nondegenerate hyperbolic $2k$-subspaces & $\displaystyle \lesssim n^{\frac{1}{k} \left\lceil \frac{k}{2}\right\rceil}$ \\

        $(n)$ & $\mathrm{P}\Omega_{2d}^+(q)$ & \shortstack{nondegenerate parabolic\\$(2k+1)$-subspaces, $q$ odd}  & $(\spadesuit)$ \\

        $(o)$ & $\mathrm{P}\Omega_{2d}^+(q)$ & nondegenerate elliptic $2k$-subspaces & $\displaystyle \lesssim n^{\frac{1}{k} \left\lceil \frac{k+1}{2} \right\rceil}$ \\

        $(p)$ & $\mathrm{P}\Omega_{2d}^-(q)$ & \shortstack{nondegenerate parabolic\\$(2k+1)$-subspaces, $q$ odd} & $(\spadesuit)$ \\

        $(q)$ & $\mathrm{P}\Omega_{2d}^-(q)$ & nondegenerate elliptic $2k$-subspaces & $\displaystyle \lesssim n^{\frac{1}{k} \left\lceil \frac{k+1}{2} \right\rceil}$ \\
        
        $(r)$ & $\Omega_{2d+1}(q)$ & nondegenerate hyperbolic $2k$-subspaces, $q$ odd & $\displaystyle \lesssim n^{\frac{1}{k} \left\lceil \frac{k}{2}\right\rceil}$ \\
        
        $(s)$ & $\Omega_{2d+1}(q)$ & nondegenerate elliptic $2k$-subspaces, $q$ odd & $\displaystyle \lesssim n^{\frac{1}{k} \left\lceil \frac{k+1}{2} \right\rceil}$ \\
        \bottomrule
	\end{tabularx}%
	\caption{Upper bound for the standard actions of the primitive groups of almost simple type. The asymptotic notation and the symbols $(\spadesuit)$ and $(\clubsuit)$ are explained in \cref{remark:symbols}.}%
	\label{table}%
\end{table}%
}

\begin{thmx}\label{thm:AS}
    Let $G$ be a primitive group of almost simple type in standard action on the domain $\Omega$. Then, an asymptotic upper bound for $\m(G)$ can be found in the third column of \cref{table}.
\end{thmx}
This theorem is proved in \cref{sec:AS}.
\begin{remark}\label{remark:symbols}
    If $f,g: \NN \to \mathbb{R}^+$ are two functions such that the limit of the quotient $\limsup f(n)/g(n)$ lies in the interval $[0,1]$, then we write $f \lesssim g$. Hence the expression in the last column of \cref{table} consist of the leading term of the asymptotic expansion of the cardinality of the non-self-separable subsets we construct in the proofs.
    
    Observe that in the last column of \cref{table} the symbols $(\spadesuit)$ and $(\clubsuit)$ appear. The first symbol $(\spadesuit)$ appears every time the socle of the classical group into consideration is orthogonal and the action is on nondegenerate parabolic $(2k+1)$-subspaces. In this case, we have a precise value for the asymptotic upper bound computed, which can be read from \cref{table:ND}. The symbol $(\clubsuit)$ appears whenever we encounter the action of a classical group on totally isotropic or totally singular subspaces. Although we have an explicit upper bound for $\m(G)$, which can be derived from \cref{table:TS}, the expression are so complicated that we cannot provide a precise asymptotic, even if aided by a calculator. Still, some experimentation with these functions shows that, if one fixes $k$ and allows $d$ to grow to infinity, the expected outcome is the asymptotic behaviour $\lesssim \sqrt{n}$. %Last, the symbol $(\heartsuit)$ appears in case~$(l)$: if $k=1$, in all the other cases our method gives trivially an upper bound, and we have not investigated this case by itself to reach a different conclusion.
\end{remark}

\begin{remark}\label{rem:AS}
    For most standard actions, for a fixed even value of the parameter $k$ in \cref{table}, we see that $\m(G) \lesssim \sqrt{n}$. Hence, if the degree is large enough, we can expect $\m(G)$ to be arbitrarily close to the lower bound given in \cref{thm:neumann}. For these actions, if $k$ is odd, we can note that, as $k$ gets larger, the upper bound computed in \cref{thm:AS} gets arbitrarily close to $\sqrt{n}$. Explicitly, if $s(k,n)$ is the first term of the asymptotic expansion of the upper bounds we have computed, then
    \[ \lim_k \frac{s(k,n)}{ \sqrt{n}} = 1 .\]
\end{remark}

Let us now move on to the other types of primitive groups as described in the O'Nan–Scott Theorem \cite{LiebeckPraegerSaxl1988}. In what follows, we adopt the partition into eight classes proposed in \cite{Praeger1997}. We will recall the properties of the classes we address in \cref{sec:primitive}.

We start with primitive permutation groups of simple diagonal type.
\begin{thmx}\label{thm:SD}
    Let $G$ be a quasiprimitive permutation group of simple diagonal type, let $k\ge 3$ be the number of direct simple factors of its socle, and let $\varepsilon$ be $1$ if the socle type is a group of Lie type, and $0$ otherwise. Then,
    \[ \m(G) \le 4 \left(\frac{16}{3}\right)^{\frac{k-1}{2}} n^{\frac{1}{2} + \frac{1}{k-1}} \left(\frac{1}{4k}\log_2 n\right)^\epsilon .\]
    In particular, if $\C$ is the family of all quasiprimitive groups of simple diagonal type, then
    \[ \limsup_n \frac{\F_\C(n)}{n^{\frac{1}{2} + \frac{1}{k-1}} \log_2(n) } \le \frac{1}{k} \left(\frac{16}{3}\right)^{\frac{k-1}{2}} .\]
\end{thmx}
The proof of \cref{thm:SD}, provided in \cref{sec:SD}, also holds for primitive permutation groups of type holomorph of simple and holomorph of compound, but in these cases the resulting upper bound for $\m(G)$ is meaningless as it exceeds $n$. We leave all the holomorph types out of our analysis.

We now concentrate on primitive groups of product action type and of compound diagonal type. They can be described as being the output of the construction of a primitive wreath product applied to almost simple and simple diagonal types, respectively. This fact inspired us to obtain the following general theorem (proved in \cref{sec:extension}). 
\begin{thmx}\label{thm:productAction}
    Let $G = H \wr K$ be a permutation group in product action on the domain $\Omega = \Fun (\Gamma, \Delta)$. Then,
    \[ \m(G) \le \m(H) ^{|\Gamma|} \,.\]
\end{thmx}

Therefore, we obtain these corollaries immediately, assuming \cref{prob:2} holds for primitive permutation groups of almost simple type and primitive permutation groups of simple diagonal type, respectively. 
\begin{corx}\label{cor:PA} % Does there exist a constant $C$ such that, for the infinite class $\C_{{\rm pr}}^\ast$ of all primitive permutation groups apart from alternating and symmetric groups in their natural action,
%    \[ \limsup_n \frac{\F_{\C_{{\rm pr}}^\ast}(n)}{\sqrt{n}} = C ?\]
    Suppose that there exists a constant $C$ such that the class $\mathcal{C}_{\rm as}^\ast$ of all the primitive permutation groups of almost simple type of non alternating socle we have \[ \limsup_n \frac{\F_{\C_{{\rm pr}}^\ast}(n)}{\sqrt{n}} = C.\] Then, for the infinite class $\C_{\rm pa}^\ast$ of all primitive permutation groups of type product action whose socle is not isomorphic to a direct product of alternating groups,
    \[ \limsup_n \frac{\F_{\C_{\rm pa}^\ast}(n)}{\sqrt{n}} \le C .\]
\end{corx}

\begin{corx}\label{cor:CD} Suppose that there exists a constant $C$ such that the class $\C_{\rm sd}$ of all the primitive permutation groups of type simple diagonal type we have \[ \limsup_n \frac{\F_{\C_{{\rm pr}}^\ast}(n)}{\sqrt{n}} = C.\]
    Then, for the infinite class $\C_{\rm cd}$ of all primitive permutation groups of type compound diagonal,
    \[ \limsup_n \frac{\F_{\C_{\rm cd}}(n)}{\sqrt{n}} \le C .\]
\end{corx} 

This concludes the extended summary of our work, while the rest of the paper is devoted to the proofs.

\section{Properties of $\m(G)$}
In this section, we investigate the monotonicity of $\m(G)$, $\f_\C(n)$ and $\F_\C(n)$ under certain natural operations applied either to the group $G$ or to the family $\C$.

\subsection{Monotonicity with respect to subgroups and subfamilies}
Recall that $\C$ is an infinite family of transitive permutation groups, and that $\C_n$ is the subfamily of permutation groups of degree $n$.

Observe that, for every $G\in \C_n$, every subset $A$ of the domain of $G$ satisfying $\vert A \vert < \f_{\C}(n)\leq \m(G)$ is self-separable for $G$ by definition. Therefore, to find a lower bound 
$a(n)+1$ for $\f_{\C}(n)$, we need to prove that, for all $G\in \C_n$ and for every subset $A$ in the domain of $G$ such that $\vert A\vert\leq a(n)$, $A$ is self-separable for $G$. Similarly, every subset $A$ of the domain of $G$ with the property that $|A| \ge \F_{\C}(n) \ge \m(G)$ is not self-separable. Hence, to find an upper bound $b(n)+1$ for $\F_{\C}(n)$, we need to prove that, for all $G\in \C_n$, there exists a subset $A$ in the domain of $G$ such that $\vert A\vert\geq b(n)$ and $A$ is not self-separable for $G$, that is, for every $g\in G$, we can prove that $|A\cap A^g| \ge 1$.

\begin{lemma}\label{lemma:subgroup}
    Let $H\leq G$ be two transitive permutation groups. Then
    \[\m(H) \leq \m(G) .\]
\end{lemma}
\begin{proof}
    Let $A$ be a set such that $\vert A\vert \leq \m(H)$. By definition of $\m(H)$, there exists a permutation $h \in H$ such that $A\cap A^h$ is empty. Since $h \in G$, the bound follows.
\end{proof}

Finally, recall that from \cref{rem:sub} that for families $\D\subseteq \C$ and for every positive integer $n\in X_\C$,
\[\f_\C(n) \le \f_\D(n) \le \F_\D(n) \le \F_\C(n) .\]

\begin{comment}
\color{magenta} 
*** The following two lemmas are silly and are now mentioned as
Remark~\ref{rem:sub}. Maybe we can leave them out, or at least merge them into a single one. But I vote for leaving them out. ***
\color{black}
However, for a class of permutation groups that is contained in a larger class the inequality flips for $\f_{\bullet}(n)$ (where $\bullet$ denotes a generic family of transitive groups).

\begin{lemma}
    Let $\D\subseteq \C$ be two families of transitive permutation groups. Then \[\f_\C(n)\leq \f_\D(n) .\]
\end{lemma}
\begin{proof}
    Choose $G\in \C$ such that $\m(G) = \f_\C(n)$ and $H\in \D$ such that $\m(H) = \f_\D(n).$ As $H\in \C$, by the definition of $\f_\C(n)$, we have that $\m(G)\leq \m(H)$, so the result follows.
\end{proof}

On the other hand, $\F_\bullet(n)$ has the expected behaviour.
\begin{lemma}
    Let $\D\subseteq \C$ be two families of transitive permutation groups. Then \[\F_\D(n)\leq \F_\C(n) .\]
\end{lemma}
\end{comment}

\subsection{Monotonicity with respect to quotients and extensions}\label{sec:extension}
Let $G$ be a transitive permutation group on $\Omega$, let $N$ be a normal subgroup of $G$, and let $G \act \Omega/N$ be the permutation induced by the natural action of $G$ on the orbit-space
\[\Omega/N = \left\{ \omega^N \mid \omega \in \Omega \right\} .\]
\begin{lemma}\label{lem:quotient}
    Let $G$ be a transitive permutation group on $\Omega$, and let $N$ be a normal subgroup of $G$.
    Then \[ \m(G) \geq \m(G \act \Omega/N).\]
\end{lemma}
\begin{proof}
    Let $A$ be a subset of $\Omega$ such that $|A| \leq \m(G \act \Omega/N) - 1$. Then $A^N$ is a subset of $\Omega/N$ such that $|A^N| \leq  \m(G \act \Omega/N) -1 $. Hence, there exists $g \in G$ such that $A^N \cap (A^N)^g$ is empty. Therefore, $A$ is self-separable for $G$, which proves the lemma.
\end{proof}

%Since there is no clear notion of an extension for permutation groups, to mirror \cref{lem:quotient}, we
In the case of permutation groups, we focus on examining how the parameter of interest relates between a wreath product $H \wr K$ and its constituents $H$ and $K$. We start with the imprimitive action for $H\wr K$. Explicitly, if $\Delta$ is the domain of $H$ and $\Gamma$ of $K$, $H\wr K$ acts on the Cartesian product $\Delta \times \Gamma$ by 
\[ (\delta,\gamma)^{(h_1,h_2,\dots,h_{|\Gamma|})k} = (\delta^{h_\gamma}, \gamma^k),\]
where $\delta \in \Delta$, $\gamma \in \Gamma$, $(h_1,h_2,\dots,h_{|\Gamma|}) \in \Fun(\Gamma,H)$ and $k\in K$. (Recall that $\Fun(\Gamma,H)$ is the group of functions from $\Gamma$ to $H$, and it can be identified with the $|\Gamma|$-fold direct product $H^{|\Gamma|}$.)

%Now, we prove \cref{thm:boundsTransitive}.
%Recall the statement of the theorem. 
%\begin{thmnon}[\cref{thm:boundsTransitive}]
%    Let $G$ be a group, and suppose that $G$ embeds into the imprimitive wreath product $H \wr K$. Then
%    \[ \m(K) + \m(H) - 1 \le \m(G) \le \m(K) \m(H) \,.\]
%\end{thmnon}
\begin{proof}[Proof of \cref{thm:boundsTransitive}]
Let $G$ be a transitive permutation group that embeds in $H \wr K$ with imprimitive action, acting on $\Delta \times \Gamma$ as described above. For each element of $\gamma \in \Gamma,$ we will use the notation $\Delta ^\gamma =\{(\delta,\gamma)| \delta \in \Delta\}.$ Note that the set of disjoint copies of $\Delta$ form a block system for $G,$ and we will refer to these copies as blocks in the proof.

We start by proving the lower bound. Let $A$ be any set of points of size $\m(H)+\m(K)-2$. Either, there exists $\gamma \in \Gamma$ such that $\Delta^\gamma \cap A\geq \m(H)$ points or not. If the latter holds, then there is an element $g$ of the base group $\Fun(\Gamma,H)$, such that $A$ and $A^g$ intersect trivially. On the other hand, if there is such $\gamma$, then $A$ has a nontrivial intersection with at most $\m(K)-1<\frac{|\Gamma|}{2}$ copies of $\Delta$, hence there exists an element $g\in K$ from the top group that maps the blocks intersecting $A$ to blocks that do not. In any case, we are able to find a permutation $g$ such that $A\cap A^g$ is empty, and hence every set of size $\m(H)+\m(K)-2$ is self-separable for $G$. The lower bound immediately follows. 

For the upper bound, consider a subset of points
\[ A = \{ (x,y) \mid x\in X\subseteq \Delta, y\in Y\subseteq \Gamma, 1\le X \le \m(H) ,\, 1\le Y \le \m(K)\} \,,\]
where we identified the domain of $G$ with the Cartesian product $\Delta \times \Gamma$. By the definition of $\m(H)$ and $\m(K)$, it follows that, for every $g\in G$, $|A\cap A^g|\ge 1$. This is enough to prove the upper bound.
\end{proof}

Now, we focus on wreath products in product action. The wreath product of $H \wr K$ can be endowed with an action on $\Fun(\Gamma, \Delta)$, which we understand as the $|\Gamma|$-fold Cartesian power of $\Delta$. Indeed, for every $g = (h_1, h_2, \ldots, h_{|\Gamma|}) k \in H \wr K$, for every $f\in \Fun(\Gamma, \Delta)$ and for every $\gamma \in \Gamma$, we let
\[ f^g(\gamma) = f \left(k^{-1}\gamma \right)^{h_{k^{-1}\gamma}}.\]
(For ease of reading, we are denoting the action of $K$ on $\Gamma$ as a left multiplication.)%Now we prove \cref{thm:productAction}. Recall the statement of the theorem. 
%\begin{thmnon}[\cref{thm:productAction}]
%       Let $G = H \wr K$ be a permutation group in product action on the domain $\Omega = \Fun (\Gamma, \Delta)$. Then,
%    \[ \m(G) \le \m(H) ^{|\Gamma|} \,.\]
%\end{thmnon}

\begin{proof}[Proof of \cref{thm:productAction}]
    Let $R \subseteq \Delta$ be a subset with $|R|=\m(H).$ Hence, for every $h\in H$, $R\cap R^h$ is not empty, and consider
    \[A = \Fun (\Gamma,R) = \left\{ f \in \Fun(\Gamma, \Delta) \mid  f(\gamma)\in R \right\} \forall \gamma\in \Gamma .\]
    For every $g = (h_1, h_2, \ldots, h_{|\Gamma|}) k \in H \wr K$,  every $f\in A$ and every $\gamma \in \Gamma$, we see that
    \[ f^g(\gamma) = f \left(k^{-1}\gamma\right)^{h_{k^{-1} \gamma}} \in R^{h_{\gamma k^{-1}}} . \]
    We now observe that
    \[ A \cap A^g = \Fun (\Gamma,R) \cap \left( R^{h_{1 k^{-1}}} \times R^{h_{2 k^{-1}}} \times \ldots \times R^{h_{|\Gamma| k^{-1}}} \right)\]
    is not empty due to the choice of $R$. In particular, by applying \cref{lemma:subgroup} we obtain that
    \[ \m(G) \le \m(H) ^ {|\Gamma|} . \qedhere \]
\end{proof}

\subsection{Monotonicity with respect to change of action}
We complete the section by establishing a connection between the self-separability parameter for two different actions of the same abstract groups.

\begin{lemma}\label{lemma:nestedstabilsers}
    Let $G$ be an abstract group, and let $H$ and $K$ be two core-free subgroups of $G$ with $H\le K$. Then
    \[ \frac{\m (G \act G/H)}{|K:H|} \le \m(G \act G/K) \le \m (G \act G/H) .\]
\end{lemma}
\begin{proof}
    Let $m$ be a positive integer such that $m \le \m (G \act G/K) -1$, and let $A$ be an $m$-subset of $G/H$, say
    \[ A = \{Ha_1,Ha_2,\dots,Ha_m\} .\]
    Without loss of generality, we can define
    \[ A' = \{Ka_1,Ka_2,\dots,Ka_k\} ,\] such that for all $i' \in \{1,2,\dots,m\}$, there exists $i \in \{1,2,\dots,k\}$ such that $Ha_{i'} \subset Ka_i.$
    Now $A'$ is a $k$-subset of $G/K$, for some $k\le m$. By our choice of $m$, there exists a group element $g\in G$ such that $A'\cap A'g$ is empty, or, more explicitly, such that
    \[ \left(\bigcup_{i=1}^k Ka_ig\right) \cap \left(\bigcup_{j=1}^k Ka_j\right) \quad \hbox{is empty}\,.\]
    Since, for every $i' \in \{1,2,\dots,m\}$, there exists $i \in \{1,2,\dots,k\}$ such that $Ha_{i'} \subset Ka_i$, and since the right multiplication by $g$ is a bijection on $G/H$, it follows that
    \[  \left(\bigcup_{i=1}^m Ha_ig\right) \cap \left(\bigcup_{j=1}^m Ha_j\right) \quad \hbox{is empty}\,.\]
    In particular,
    \[m \le \m (G \act G/H) -1 ,\]
    and, by taking the maximum on both sides, we establish the upper bound for $\m(G \act G/K)$.
    
    Let now $h$ be a positive integer such that $h \ge \m(G \act G/K)$, and let $B$ be a $h$-subset of $G/K$, say
    \[ B = \{Kb_1,Kb_2,\dots,Kb_h\} .\]
    Let $\{k_1,k_2, \dots, k_{|K:H|}\}$ be a set of representative for the right cosets of $H$ in $K$, and consider 
    \[ B' = \{Hk_1b_1,Hk_1b_2,\dots,Hk_1b_m,Hk_2b_1,\dots, Hk_{|K:H|}b_h\}.\]
    Observe that $|B'|\ge h |K:H|$, and that
    \[ \bigcup_{i=1}^h Kb_i = \bigcup_{i=1}^h \left( \bigcup_{j=1}^{|K:H|} Hk_jb_i \right).\]
    Since $B$ is not self-separable for $G$, the previous equality implies that $B'$ is not self-separable for $G$. Therefore,
    \[ h |K:H| \ge \m(G \act G/H).\]
    By taking the minimum on both sides, we obtain the desired lower bound.
    
    Hence, the proof of \cref{lemma:nestedstabilsers} is complete.
\end{proof}

We would like to point out that \cref{lemma:nestedstabilsers} is quite weak in a general setting. For instance, consider the symmetric group $\Sym(n)$ in its natural action (that is, with stabiliser $\Sym(n-1)$) and in its regular action (that is, with a trivial stabiliser). Then, applying the upper bounds from
\cref{lemma:nestedstabilsers,thm:regular}, we deduce that
\[ \m \left( \Sym(n) \act [n] \right) \le \m \left( \Sym(n) \act \Sym(n) \right) \le \frac{4}{\sqrt3} \sqrt{n!} \,.\]
The upper bound is comically ineffective, as the domain of $\Sym(n) \act [n]$ contains only $n$ points. This is caused by the fact that the order of the stabiliser $\Sym(n-1)$ is enormous. 

On the other hand, if the index $|K:H|$ can be bounded by a constant, \cref{lemma:nestedstabilsers} can produce significant estimates. (For two functions $f,g: \NN \to \mathbb{R}^+$, we write $f \asymp g$ if there exists a positive integer $M$ and a positive constant $C$ such that, for every $n\ge M$, their quotient $f(n)/g(n)$ is positive and bounded from above by $C$.)

\begin{lemma}\label{lemma:familyBounded}
    Let $\C$ and $\D$ be two families of permutation groups such that, for every abstract group $G$, if $G$ appears in $\C$ with stabiliser $K$, then $G$ appears in $\D$ with stabiliser $H$. Suppose that there is a constant $d$ such that, for every $G$, $|K:H|\le d$. Then
    \[ \f_\C \asymp \f_\D \quad \hbox{and} \quad \F_\C \asymp \F_\D .\]
    
\end{lemma}
\begin{proof}
The claimed results are obtained by taking the minimum and the maximum of the inequalities from \cref{lemma:nestedstabilsers} and observing that a linear dilatation of the argument of the function does not change the asymptotic behaviour.
\end{proof}

\begin{comment}
    We set $n = |G:H|$. By \cref{lemma:nestedstabilsers}, we have that, for every $G$,
    \begin{equation}\label{eq:familyBounded}
        \f_\C(n) \le \m (G \act G/K) \le \m (G \act G/H) .
    \end{equation}
    We define the auxiliary function
    \[ X(n) := \max\limits_{n \le m \le nd} \f_\D(m) .\]
    By taking the minimum on both sides of \cref{eq:familyBounded},
    \[ \f_\C(n) \le \f_\D(n|K:H|) \le X(n) .\]
    By construction, $\f_\D \asymp X$. Hence, there exists a positive constant $C$ such that
    \[ \f_\C(n) \le C \f_\D(n) .\]
    On the other hand, \cref{lemma:nestedstabilsers} also implies
    \begin{equation}\label{eq:familyBounded2}
        \frac{1}{d}\f_\D(n) \le \frac{1}{d}\m (G \act G/H) \le \m (G \act G/K) .
    \end{equation}
    ????
    
    Therefore, we obtain
    \[ 0 < \liminf_n \frac{\f_\C(n)}{\f_\D(n)} \le \limsup_n \frac{\f_\C(n)}{\f_\D(n)} \le C, \]
    and thus $\f_\C \asymp \f_\D$.

    The proof that $\F_\C \asymp \F_\D$ is analogous to the previous one. The differences lies in the fact that the starting point is now
    \[ \m (G \act G/K) \le \m (G \act G/H) \le \F_\D(n) , \]
    and the approapriate auxiliary function is
    \[ Y(n) := \min\limits_{n/d \le m \le n} \F_\C(m) \qedhere.\]
\end{comment}

\section{Proof of \cref{thm:neumann}} \label{sec:lowerBound}
We give a proof of Neumann's separation lemma (see \cite{BirchBurnsMacdonaldNeumann1976,Neumann1975} for the seminal version of this result, and \cite[Theorem~5.3]{AraujoCameronSteinberg2017} for our version), and we derive the lower bound in \cref{thm:neumann} as a corollary.

\begin{lemma}\label{lemma:separationNeumann}
    Let $G$ be a transitive permutation group on $\Omega$, and let $A$ and $B$ be two subsets of $\Omega$. Then
    \[ \frac{1}{|G|} \sum_{g\in G} |A\cap B^g| = \frac{|A||B|}{|\Omega|} .\]
\end{lemma}
\begin{proof}
    We define the set
    \[ X= \left\{ (\alpha,\beta,g) \in \Omega \times \Omega \times G \mid \alpha\in A, \beta \in B, \alpha = \beta^g \right\} \,.\]
    For every $\alpha,\beta \in \Omega$, the set of elements $g \in G$ satisfying $\alpha = \beta^g$ is a right coset of the stabiliser $G_\alpha$. Therefore,
    \begin{equation}\label{eq:NeumannDouble1}
        |X| = |A||B||G_\alpha| = \frac{|A||B||G|}{|\Omega|} \,.
    \end{equation}
    On the other hand, for every $g$, we can count the number of pairs $(\alpha,\beta)$ such that $\alpha = \beta^g,$ which is equal to $|A\cap B^g|.$ Hence,
    \begin{equation}\label{eq:NeumannDouble2}
        |X| = \sum_{g\in G} |A\cap B^g| \,.
    \end{equation}
    The result is obtained by equating \cref{eq:NeumannDouble1,eq:NeumannDouble2}.
\end{proof}

\begin{proof}[Proof of \cref{thm:neumann}]
    Let $G$ be a transitive permutation group on $\Omega$ of degree $n$, and let $A$ be a subset of $\Omega$. By \cref{lemma:separationNeumann}, we have
    \begin{equation}\label{eq:neumann}
        \frac{1}{\vert G\vert}\sum_{g\in G}\vert A \cap A^g\vert =\frac{\vert A\vert ^2}{n} \,.
    \end{equation}
    If there does not exist $g\in G$ such that $A\cap A^g$ is empty, then
    \[\frac{1}{\vert G\vert}\sum_{g\in G}\vert A \cap A^g\vert \ge \frac{1}{\vert G\vert}\sum_{g\in G \setminus\{1\}} 1 + \frac{|A|}{|G|} = 1 + \frac{|A|-1}{|G|}
    %\geq 1 \,
    ,\]
    where equality is attained if and only if $|A\cap A^g|=1$ for every nontrivial $g \in G$.
    
    By \cref{eq:neumann}, we also see that
    \[\frac{\vert A\vert ^2} {n} \geq 1 + \frac{|A|-1}{|G|}, \]
    and thus solving the quadratic inequality gives
    \[\vert A\vert \geq \frac{1 + \sqrt{1 - 4|G|(1-|G|)/n}}{2|G|/n} = \frac{1}{2|G_\alpha|} + \sqrt{n - \frac{1}{|G_\alpha|} + \frac{1}{4|G_\alpha|^2}}\, .\]
    Let us denote the expression on the right-hand side  by $\alpha(G)$. We have shown that, if a subset $A$ is not self-separable for $G$, its size needs to be at least $\alpha(G)$. Therefore, $\m(G) \ge \alpha(G)$, as required.

    Now, we shall prove that the lower bound is met if and only if $\Omega$ is the point-set of a projective plane and $G$ is a regular group on $\Omega$. Let 
    \[ A^G = \{ A^g \mid g \in G\},\]
    and consider the incidence structure (or design) $(A^G, \Omega)$, where $A^g$ is incident with $\omega$ if $\omega \in A^g$. Observe that $G$ preserves this incidence relation.
    
    Recall that, if $|A|=\alpha(G)$, then $|A\cap A^g|=1$ for every nontrivial permutation $g \in G$. Thus, the set-wise stabiliser $G_A$ of $A$ is trivial, and hence the action of $G$ is regular on $A^G$. The transitivity of $G$ on $\Omega$ and $A^G$ gives us that $(A^G, \Omega)$ is uniform and regular. Clearly, the design is incomplete, and the hypothesis that, for $g\in G \setminus \{1\}$ $|A\cap A^g|=1$ implies that it is balanced. Therefore, $(A^G, \Omega)$ is a balanced incomplete block design, and we can apply Fisher's inequality. Doing so, we obtain
    \[ |G| = |A^G| \le |\Omega| = n .\]
    Therefore, $G$ is regular on $\Omega$, and hence the design is symmetric. Using our hypothesis on the intersections of sets in $A^G$, $(A^G, \Omega)$ is a projective plane, and so is its dual $(\Omega,A^G)$.
\end{proof}

We would like to note here that, if $G$ acts regularly on a non-degenerate projective plane of order $q$, then the right-hand side of Inequality~\eqref{eq:lowerBound}  in \cref{thm:neumann} is an integer. In fact,
\[ \alpha (G) = \frac{1}{2} + \sqrt{q^2 + q + 1 - 1 + \frac{1}{4}} = \frac{1}{2} + \sqrt{\left( q + \frac{1}{2} \right) ^2 } = q+1  .\]
Furthermore, we observe that the equality can hold if and only if such right-hand side is an integer, which in turn happens if and only if $n$ is of the form $n = (m|\alpha|)^2 + (2|G_\alpha|-1)m + 1$ for some positive integer $m$. When the right-hand side is not an integer, then the lower bound for $\m(G)$ can be improved by taking the ceiling  of  that value. This makes us pose the following: 

\begin{vprasanje}
    Does there exist a transitive permutation group $G$ of degree $n$ not of the form $(m|G_\alpha|)^2 + (2|G_\alpha|-1)m + 1$ for any $m \in \NN$ such that
    \[\m(G) = \left\lceil  \left( \, \frac{1}{2|G_\alpha|} + \sqrt{n - \frac{1}{|G_\alpha|} + \frac{1}{4|G_\alpha|^2}}  \, \right)  \right \rceil ?\]
    If not, how can the bound in \cref{thm:neumann} be improved in this case?
\end{vprasanje}

\section{Proof of \cref{thm:regular}}\label{sec:reg}
We start this section by focusing our attention on groups acting regularly. Then, we will generalize our result on regular groups to families of permutation group with stabilisers of bounded order.

\begin{remark}\label{rem:regular}
    Let $G$ be a regular permutation group of degree $n$ on $\Omega$. As usual, we identify $G$ with $\Omega$, so that the action we are dealing with is the action of $G$ on itself by right multiplication. Let $A$ be a subset of $G$.
    Suppose that there is a $g\in G$ such that $A \cap Ag$ is empty. Then $g$ is not an element of $A^{-1}A$, and thus $A^{-1}A$ is a proper subset of $G$.
    On the other hand, if $A \cap Ag$ is nonempty for every $g\in G$, then $A^{-1}A=G$. Therefore, $A$ is not self-separable  if and only if $A^{-1}$ is a difference basis (see \cref{de:differenceBasis}). Hence, $\m(G)$ is %the size of the smallest set $A$ such that $A^{-1}A=G$, that is,
    the size of the smallest difference basis for $G$. 
\end{remark}

A notion closely related to that of difference basis appears in additive combinatorics, introduced in \cite{RedeiRenyi1949} by R\'edei and R\'enyi, where they consider subsets $A \subseteq \mathbb{Z}$ such that the difference set $A - A$ covers the finite subset $\{0,1,2,\dots, n\}$. The problem of constructing such sets $A$ of minimal possible size remains a relevant question in additive number theory (see, for instance, \cite{BernshteynTait2019,Schoen2007}).

These ideas also intersect naturally with classical topics in combinatorics and finite geometry. For instance, \emph{planar difference sets} are difference bases $A \subseteq G$ with the additional property that, for every $g \in G$, there exists a unique pair $(a,b) \in A^2$ such that $g = ab^{-1}$. Such sets play a fundamental role in the construction of finite projective planes. In particular, in his seminal work \cite{Singer1938}, Singer showed that every finite projective plane admits a cyclic collineation group of order equal to the number of its points. Observe that this fact is directly connected to the sharpness of the lower bound in \cref{thm:neumann}.

Another related concept is that of a \emph{basis for a group $G$}: a subset $A \subseteq G$ such that $A^2 = G$. If $A$ is inverse-closed, then the notions of basis and difference basis coincide. This concept originates with Rohrbach \cite{Rohrbach1937}, and has since been developed further in \cite{BertramHerzog1991,KozmaLev1992,Nathanson1992}. We would like to note that, as observed in \cite{MacbethSiagiovaSiran2012}, the existence of small bases is related to the \emph{degree-diameter problem} for Cayley graphs of diameter $2$.

The cardinality of a minimal difference basis is bounded below by the size of a minimal difference set, and an upper bound of $3\sqrt{|G|}$ is given in \cite{FinkelsteinKleitmanLeighton1988}. Determining the exact minimal size of a difference basis remains an active area of research, even for specific families of groups (see \cite{BanakhGavrylkiv2019Cyclic,BanakhGavrylkiv2019Dihedral,BanakhGavrylkiv2019Abelian}). Any improvement to these bounds would lead to immediate improvements in the asymptotic estimates computed in \cref{thm:regular}. For example, for any finite Abelian $p$-group $G$ in its regular right action, where $p$ is a prime with $p \geq 11$, the result of \cite{BanakhGavrylkiv2019Abelian} implies that
\[
\m(G) < \frac{\sqrt{2}(\sqrt{p} - 1)}{\sqrt{p} - 3} \sqrt{|G|}.
\]

\begin{proof}[Proof of \cref{thm:regular}]
    Let $G$ be a finite group, and suppose that $G$ acts regularly on itself by right multiplication. By \cref{rem:regular}, we need to build a difference basis $A$ for $G$ of as small size as possible.
    By \cite[Theorem~1]{KozmaLev1992}, there exist two positive constants $c_1$ and $c_2$ and two subsets $B$ and $C$ such that
    \[|B| = c_1 \sqrt{|G|} , \quad  |C| = c_2 \sqrt{|G|}, \quad c_1 + c_2 \leq 4/\sqrt3, \quad BC = G .\]
    Hence, the subset $A = B \cup C^{-1}$ is a difference basis containing at most $4\sqrt{|G|}/\sqrt3$ elements. Indeed,
    \[AA^{-1} \subseteq G = BC \subseteq (B \cup C^{-1}) (B^{-1} \cup C) = AA^{-1}. \]
    Therefore,
    \[\m(G) \leq  \frac{4}{\sqrt{3}}\sqrt{|G|} .\]
    
    Now we can compute the limit for $\f_C(n)$ using \cref{thm:neumann} and the bound for the limit for $\F_\C(n)$ follows immediately.
\end{proof}

Surprisingly, \cite[Theorem~1]{KozmaLev1992} relies on the fact that every group which is not cyclic of prime order contains a subgroup $H$ whose order exceeds $|G|^{1/2}$ (see \cite{Lev1992} for a proof of this fact), which in turn relies on the classification of finite simple groups. In a classification free context, if $G$ has even order, the best result known is that there exists a subgroup $H$ with $|H|\ge |G|^{1/3}$, as the famous result of Brauer and Fawler in \cite{BrauerFowler1955} states. On the other hand, if $G$ is of odd order, then, by the Feit--Thompson Theorem \cite{FeitThompson1963}, $G$ is solvable, and solvability is enough to prove the existence of a subgroup of order exceeding the square root of the order of the group.

We can reproduce the proof of \cite[Theorem~1]{KozmaLev1992} under the additional assumption that all groups involved are solvable. In this setting the argument no longer relies on the classification of finite simple groups, and it still yields the same constant $4/ \sqrt 3$, which is sharp for the cyclic group $C_3$. Hence, \cref{thm:regular} remains valid for solvable groups in the classification free context. This prompts us to ask the following question.

\begin{vprasanje}
    Without using the classification of finite simple groups, can we find a function $\mathbf{g}: \NN \to \mathbb{R}$ such that, for every nonsolvable regular permutation groups $G$ of order $n$,
    \[ \m(G) \le \mathbf{g}(n) ?\]
\end{vprasanje}

We now turn our attention to family of transitive permutation groups whose stabilisers are bounded.
\begin{remark}\label{remark:familyBounded}
    Observe that, by applying \cref{lemma:familyBounded} with $\D$ equal to the family of regular permutation groups, we obtain that, for every $\C$ family of transitive permutation groups with stabilisers of bounded order,
    \[\f_\C \asymp \F_\C \asymp \sqrt{n} .\]
    Actually, since we have an explicit upper bound for $\f_D$ and $\F_\D$, an immediate computation shows that, if $\mathbf{c}$ is the maximum order of a stabiliser in $\C$, then
    \[ \sqrt{n} \le \f_\C(n) \le \F_\C(n) \le 4\sqrt{\frac{\mathbf{c}n}{3}} .\]
    We do not focus on this precise upper bound because, for the following application, the precise value of $\mathbf{c}$ is not known.
\end{remark}

\begin{proof}[Proof of \cref{cor:subdegree}]
    We want to apply \cref{lemma:familyBounded} to the families
    \begin{align*}
        \C &= \left\{ G  \mid G \hbox{ is a primitive permutation groups with minimal subdegree } d \right\} \,,\\
        \D &= \left\{ G \mid G \hbox{ is regular and it appears as an abstract group in } \C \right\} \,.
    \end{align*}
    By the positive solution to Sims's Conjecture in \cite{CameronPraegerSaxlSeitz1983}, there exists a function $\mathbf{g}$ such that the order of any stabiliser of a permutation group in $\C$ is at most $\mathbf{g}(d)$. Therefore, combining Theorems~\ref{thm:neumann} and~\ref{thm:regular} with \cref{remark:familyBounded}, we find that there exists a constant $\mathbf{c}(d) > 0$ such that, for every
    \[ \sqrt{n} \le \f_\C(n) \le \F_\C(n)  \le \mathbf{c}(d)\sqrt{n} .\]
    The claims of \cref{cor:subdegree} immediately follow.
\end{proof}

The proof of \cref{cor:graphRestrictive} follows \emph{verbatim} that of \cref{cor:subdegree}, where  $\C$ is the family of $G$ which are $2$-arc-transitive groups of automorphisms of some graphs $\Gamma$, whose valency do not exceed $d$, and the positive solution to Sims's Conjecture is substituted with the positive solution to Weiss's Conjecture for $2$-arc-transitive graphs obtained in \cite{Gardiner1976,Trofimov1990,Trofimov1991,Weiss1979,Weiss1993}. 

This result can be actually generalized to a larger class of automorphism groups of graphs. Given a graph $\Gamma$ and a vertex-transitive group of automorphisms $G$, the \emph{local group $L$ of the pair $(\Gamma,G)$} is the permutation group that a vertex-stabiliser induces on the neighbourhood of its fixed point. A permutation group $L$ is \emph{graph-restrictive} if there exists a constant such that, for every pair $(\Gamma,G)$ that witness $L$ as a local group, then the order of a vertex-stabiliser is bounded by this constant. For instance, the property we have used above is that $2$-transitive permutation groups are graph-restrictive. Therefore, our proof actually extends beyond the scope of \cref{cor:graphRestrictive}. For instance, in \cite{PotocnikSpigaVerret2012}, it is conjectured that any \emph{semiprimitive} permutation group is graph-restrictive. For the curious reader, examples of graph-restrictive permutation groups can be found in \cite{GiudiciMorgan2014,GiudiciMorgan2015,Morgan2023,MorganSpigaVerret2015,Spiga2016,Trofimov2022}.

\section{Proof of \cref{thm:largeBlocks}} \label{sec:ingredient}
As stated in the introduction, to discuss the sharpness of the bounds in \cref{thm:boundsTransitive} we need to understand $\m(\Sym(a) \wr \Sym(b))$, where the wreath product of symmetric groups is endowed with its imprimitive action.

\begin{proof}[Proof of \cref{thm:largeBlocks}]
    Let $G= \Sym(a)\wr \Sym(b)$, and let $\Sigma$ be its block system. We denote by $B_1,\dots, B_b$ the blocks of $\Sigma$. First, we show that $\m(G) \le a + b -1$. Consider any subset of the domain of the form
    \[A = B_1 \cup \{\omega_2\} \cup \dots \cup \{\omega_b\},\]
    where $\omega_i\in B_b$ for $2\leq i \leq a$. Note that $|A| = a + b -1$. Observe that, for every permutation $g \in G$, the set $A^g$ contains the block $B_1^g$, and, as $A$ intersects every block in at least one point, $A$ and $A^g$ have nonempty intersection. Hence, $A$ is not self-separable for $G$, and the claim follows.

    From here on, the proof is divided in multiple cases. First, we prove by induction on $b$ that, if $b$ is even or if $b\ge 5$ is odd and $a \neq 3$, then $\m(G) = a + b -1$. The cases with $a$ or $b$ equal to $3$ are delicate, hence we need to treat them separately.

    \smallskip
    \noindent \textsc{Assume that $b$ is even.}  Let $A$ be a subset of $[a]\times[b]$ of size not exceeding $a+b-2$. Our base case is $b=2$. Let $\ell = |A \cap B_1|$, and note that $|A \cap B_2| \leq a-\ell$. Now by $a$-transitivity of $\Sym(a),$ there exists $g \in \Sym(a) \wr C_2$ that switches the blocks $B_1$ and $B_2$ such that $A \cap B_1 \cap (A \cap B_2)^g$ and $A \cap B_2 \cap (A \cap B_1)^g$ are both empty. Thus, $A$ and $A^g$ have trivial intersection, and $A$ is self-separable for $G$.

    Now, assume that the property holds for $b$, and we claim that it also holds for $b + 2$. We again have two cases to consider. 
    
    If there does not exist a block $B \in \Sigma$, such that $|B \cap A| \ge \lceil (a+1) /2 \rceil$, then, there exists an element $g$ of the base group $\Sym(a)^b$ such that, for every $B\in \Sigma$, $|A\cap B \cap (A\cap B)^g| = 0$. Therefore, $A$ is self-separable for $G$.
    
    Otherwise, without loss of generality, we suppose that $|B_1 \cap A| = \ell \ge \lceil (a+1) /2 \rceil$. We aim to prove the existence of a block $B \in \Sigma$ such that $|B \cap A| \leq a - \ell$. Aiming for a contradiction, we assume that, for every $2\leq i \leq a$, $|B_i \cap A| \ge a - \ell +1$. We compute 
    \begin{align*}
        |A| &= \sum_{B_i \in \Sigma}|B_i \cap A| 
        \\&\geq \ell + (b-1)( a -\ell + 1)
        \\&= (b-1)(a+1) - \ell(b-2)
        \\&\geq (b-1)(a+1) - a(b-2)
        \\&= a + b - 1 .
    \end{align*}
    This goes against the assumption that $|A| \le a + b - 2$. Therefore, without loss of generality, we can assume that $ |B_2 \cap A| \leq a - \ell$. We define
    \[ A_0 = A \cap (B_1 \cap B_2) \quad \hbox{and} \quad A_1 = A \cap \left( \bigcup_{i=3}^b B_i \right) ,\]
    and we observe that $|A_0| \le a $ and that $|A_1| \le a + b - \ell - 2 \le a +b - 4$. We can apply the inductive hypothesis on the pairs $(A_0, \Sym(a) \wr C_2)$ and $(A_1, \Sym(a) \wr \Sym(b-2))$, thus obtaining that there exists a permutation $g \in (\Sym(a) \wr C_2) \times (\Sym(a) \wr \Sym(b-2)) \le \Sym(a) \wr \Sym(b)$ such that $A \cap A^g$ is empty. Hence, $A$ is self-separable for $G$, and the proof of this case is complete.

    \smallskip
    \noindent \textsc{Assume that $b \ge 5$ is odd and that $a\neq 3$.} Let $A$ be a subset of $[a]\times[b]$ of size not exceeding $a+b-2$. By the restraint on $b$, our base case is $b=5$. We claim that $|A \cap B_i| > a/2$ cannot happen for three blocks at once. Without loss of generality, assume we have $|A \cap B_i| > a/2$ for $i=1,2$ and $3.$ Then
    \[ |A \cap (B_1\cup B_2\cup B_3)| \ge 3 \left\lceil \frac{a+1}{2} \right\rceil \geq  a + \frac{7}{2}.\]
    (Observe that the last inequality does not hold if $a=3$.) This contradicts the assumption that $|A| \le a + b - 2 = a + 3$. Hence, the number of blocks containing more than $a/2$ points is at most $2$. 
    
   There are three cases to consider, corresponding to the number of blocks with $|A \cap B_i| > a/2$. 
   
   If there are no such blocks, we can find a permutation $g$ of the base group such that $A \cap A^g$ is empty, as we did for the case $b=2$. Suppose now that there exists exactly one block for which its intersection with $A$ contains more than $a/2$ points, and, without loss of generality, we assume that this block is $B_1$, and we let $\ell = |B_1 \cap A|$. We aim to show that then there exists a block $B_i$ for which $|B_i \cap A| \leq b - \ell$. Suppose for a contradiction that, for every $2 \leq i \leq 5$, $|B_i \cap A| \ge a - \ell +1.$ Then
    \begin{align*}
        |A| &= \sum_{B_i \in \Sigma}|B_i \cap A| 
        \\&\geq \ell + 4( a -\ell + 1)
        \\&= 4a - 3\ell +4
        \\&\geq a+4 .
    \end{align*}
    contradicting the cardinality assumption on $A$. So, without loss of generality, $|B_2 \cap A| \leq a - \ell$. Hence, reasoning as in the previous base case, we see that there exists an element $g$ such that $g$ stabilizes the blocks $B_3$, $B_4$ and $B_5$, it swaps the blocks $B_1$ and $B_2$ and $A \cap A^g$ is empty. A similar argument can be carried out to deal with the remaining case. Henceforth, we have proved that every subset $A$ containing less than $a+4$ points is self-separable for $\Sym(a) \wr \Sym(5)$ (and $a\neq 3$).

    The inductive step for $b$ odd (and $a \neq 3$) is similar to the one we used for $b$ even.

    \smallskip
    \noindent \textsc{Assume that $b=3$ and $a$ is odd.}
    We first show that $\m(G) \le a+b-2 = a+1$. Indeed, we choose $A$ so that
    \[\vert B_1\cap A\vert=\vert B_2\cap A\vert= \frac{a+1}{2} .\]
    Then, for every $g\in G$, $A^g$ and $A$ will intersect $B_1$ or $B_2$ in at least one point. Hence $A$ is not self-separable. Our objective is to show that every set $A$ whose cardinality is at most $a$ is self-separable for $\Sym(a) \wr \Sym(3)$.

    We claim that, up to permuting the indices of the blocks, if $\vert B_1\cap A\vert = \ell \geq (a+1)/2 $, then $\vert B_2\cap A\vert\leq a-\ell$, and $\vert B_3\cap A\vert\leq (a+1)/2 $. Aiming for a contradiction, assume that either $\vert B_2\cap A\vert\leq a-\ell$ or $\vert B_3\cap A\vert\leq (a+1)/2$ fails. If the former happens, we see that
    \[ |A| \geq |B_1 \cap A| + |B_2 \cap A| > \ell + a - \ell = a .\]
    Meanwhile, if the latter holds, then
    \[ |A| \geq |B_1 \cap A| + |B_3 \cap A| \ge 2 \frac{a+1}{2} \ge a+1.\]
    In both cases, we find a contradiction to $|A|\le a$. Using our usual arguments, our claim implies that there exists a permutation $g\in \Sym(a) \wr \Sym(3)$ such that $g$ swaps the blocks $B_1$ and $B_2$, it fixes $B_3$ and $A\cap A^g$ is empty. Therefore, we established the theorem in this setting.
    
    \smallskip
    \noindent \textsc{Assume that $b=3$ and $a$ is even.} Let $A$ be a subset of cardinality at most $a+1$. We can deal with the case with every block containing less than $a/2$ points as usual. Hence, without loss of generality, we suppose that $\vert B_1\cap A\vert = \ell \geq a/2 +1$. It cannot be that both $\vert B_2\cap A\vert \ge a - \ell +1$ and $\vert B_2\cap A\vert \ge a - \ell +1$, as that would imply that
    \[ |A| \ge \sum_{B_i \in \Sigma} |A \cap B_i| = \ell + 2(a- \ell +1) \ge a+2 , \]
   contradicting the cardinality assumption on $A$. Without loss of generality, we have that $|A \cap B_2| \le a - \ell$. Therefore, there exists a permutation $g$ fixing $B_3$ and swapping $B_1$ and $B_2$ such that $A\cap A^g$ is trivial. Thus, $A$ is self-separable for $\Sym(a) \wr \Sym(3)$, and the desired equality follows.

    \smallskip
    \noindent \textsc{Assume that $b$ is odd and $a=3$.}
    Observe that, in this setting, there exists a set of size $b+1$ which is not self-separable. Indeed, let $A$ be a subset of $[a]\times [b]$ such that, for every $i \le \lceil (b+1)/2\rceil$, $|A \cap B_i| = 2$, and $A\cap B_i$ is empty otherwise. Now, for every permutation $g \in \Sym(3) \wr \Sym(b)$, there exists an index $i$ such that
    \[ | (A \cap B_i) \cap (A \cap B_i)^g | \ge 1.  \]
    Hence $\m(G)\le b+1$.
    
    Finally, we argue by induction on $b$. The base case $b=3$ was already taken care of before, while the induction argument is the same as above.
    
    This last case exhausts the possibilities for the pairs $(a,b)$, and hence it completes the proof of the first part of the statement of \cref{thm:largeBlocks}.

    If $G$ is an imprimitive permutation group, and $\Sigma$ is a system of blocks such that $|\Sigma|=b$ and each block contains $a$ point, then $G$ embeds in $\Sym(a) \wr \Sym(b)$. The second part of the statement then follows combining the first one with \cref{lemma:subgroup}.
\end{proof}

\section{Proof of \cref{theorem:comput}} 
\label{sec:upperBound}

In this section, we prove \cref{theorem:comput}, that is, we characterise permutation groups of degree $n$ with $\m(G) = \lceil (n+1)/2\rceil$, which is the upper bound from \cref{thm:upperBound}.

%Although deriving the upper bound is as straightforward as applying \cref{rem:Sym}, establishing its sharpness is a significantly more delicate and challenging task. 
We divide the discussion into two fundamentally distinct cases: primitive groups (see \cref{sub:upperPrimitive}) and imprimitive groups (see \cref{sub:upperImprimitive}). 
%Namely, we prove the following characterisations.

\subsection{Primitive groups meeting the upper bound}
\label{sub:upperPrimitive}

The characterisation of primitive permutation groups $G$ meeting the upper bound on the parameter $\m(G)$ from 
\cref{thm:upperBound} is summarised in the following lemma.

\begin{lemma}
\label{lemma:complementB}
    Let $G$ be a primitive permutation group of degree $n$. Then $\m(G) = \lceil(n+1) /2 \rceil$ if and only if 
    one of the following holds: 
      \begin{itemize}
       \item $n\ge 2$ and $G\cong \Sym(n)$ in its natural action;
       \item $n\ge 3$ and $G\cong \Alt(n)$ in its natural action;
       \item $n=5$ and $G\cong {\rm C}_5$, $ {\rm Dih(5)}$, $ {\rm AGL}_1(5)$;
       \item $n=6$ and $G\cong \Sym(5)$ in its action on $6$ points;
       \item $n=7$ and $G\cong {\rm Dih}(7)$, ${\rm AGL}_1(7)$;
       \item $n=8$ and $G\cong {\rm AGL}_1(8)$, ${\rm A\Gamma L}_1(8),{\rm ASL}_3(2)$, ${\rm PSL}_2(7)$, ${\rm PGL}_2(7)$;
       \item $n=9$ and $G\cong {\rm AGL}_1(9)$, ${\rm PSU}_3(2)$, %$\cong M_9$,
             ${\rm A\Gamma L}_1(9)$, ${\rm ASL}_2(3)$, ${\rm AGL}_2(3)$, ${\rm PSL}_2(8)$, ${\rm P\Gamma L}_2(8)$;
      \item $n=10$ and $G\cong {\rm PGL}_2(9)$, ${\rm P\Gamma L}_2(9)$;
      \item $n=12$ and $G\cong {\rm PGL}_2(11)$, ${\rm P\Gamma L}_2(11)$, $M_{11}$, $M_{12}$;
      \item $n=14$ and $G\cong {\rm PGL}_2(13)$;
      \item $n=16$ and $G\cong {\rm AGL}_4(2)$;
      \item $n=24$ and $G\cong M_{24}$.
      \end{itemize}
\end{lemma}

The rest of the section is devoted to the proof of the above characterisation.
We start our analysis with two general observations, that will greatly reduce the possibilities for primitive permutation group meeting the bound from \cref{thm:upperBound}.

\begin{lemma}\label{lemma:PrimozRevenge}
    Let $G$ be a permutation group acting on a set $\Omega$ of size $n$, let $c(g)$ be the number of cycles in the cyclic decomposition of $g\in G$, 
    let $\ell(g)$ be the length of the shortest odd cycle in that composition (with $\ell(g) = 0$ if $g$ contains no odd cycles), 
    and for a non-negative integer $i$, let 
    $\mathcal{O}_i$ be the set of elements $g$ in $G$ whose cycle decomposition contains at most $i$ odd cycles.
    Further,
    let $r(G)$ be the number of orbits of $G$ in its induced action on ${\Omega\choose \lfloor n/2\rfloor}$.  Suppose that $\m(G) = \lceil (n+1)/2\rceil$. 
    Then, if $n$ is even,
    \[ r(G) |G|  = \sum_{A \in  {\Omega \choose n/2}}|G_A| = \sum_{g \in \mathcal{O}_0}2^{c(g)}, \]
    and, if $n$ is odd,
        \[  r(G) |G| = \sum_{A \in  {\Omega \choose \lfloor n/2\rfloor}}|G_A| \leq \sum_{g \in \mathcal{O}_1}2^{c(g)-1}\ell(g)\leq  \left \lceil \frac{n}{2}   \right \rceil \sum_{A \in  {\Omega \choose \lfloor n/2\rfloor}}|G_A| =  \left \lceil \frac{n}{2}   \right \rceil r(G) |G|. \]
\end{lemma}
\begin{proof}
Let us start by observing that the equality
\[r(G) |G| = \sum_{A \in  {\Omega \choose \lfloor n/2\rfloor}}|G_A|\]
appearing in both formulas, for $n$ odd and for $n$ even, is just a restatement of the Cauchy-Frobenius Lemma (see~\cite[Theorem~1.7A]{DixonMortimer}).

In what follows, we shall
assume without loss of generality that $\Omega = [n]$. By double counting the set
    \[ \left\lbrace (A,g) \in {[n] \choose \lfloor n/2\rfloor } \times G \>  \middle\vert \> A\cap A^g = \emptyset \right\rbrace, \]
    we see that \begin{equation}\label{eq:itsJustDoubleCounting}
    \sum_{A \in {[n] \choose \lfloor n/2\rfloor }} \left| \left\lbrace g \in G \mid A\cap A^g = \emptyset  \right\rbrace \right| = \sum_{g\in G} \left| \left\lbrace A \in {[n] \choose \lfloor n/2\rfloor} \> \middle\vert\> A\cap A^g = \emptyset  \right\rbrace \right| .
    \end{equation}
    
    First, we focus on the case when $n$ is even. By assumption, every $(n/2)$-subset $A$ is self-separable, and there is a unique image of $A$ that is disjoint from it. Hence, it follows that
    $\left\lbrace g \in G \mid A\cap A^g = \emptyset  \right\rbrace = G_Ah$,
    where $h\in G$ is any element such that $A\cap A^h = \emptyset$. Hence,
    \[ \sum_{A \in {[n] \choose \lfloor n/2\rfloor }} \left| \left\lbrace g \in G \mid A\cap A^g = \emptyset  \right\rbrace \right| = \sum_{A \in {[n] \choose \lfloor n/2\rfloor }} |G_A| .\]
    Let us now count how many $(n/2)$-subsets can be separated from themselves by a fixed $g \in G$.
    Suppose $g$ is written as a product of $c(g)$ disjoint cycles, and let $A$ be such a set. Observe that $A$ must meet every cycle of the decomposition in precisely half of the points of that cycle. In particular, every cycle has even length.
    Furthermore, if $\pi=(\pi_1,\pi_2,\dots,\pi_{2\ell})$ is one of the cycles of $g$, $A$ must contain either all the $\pi_i$ whose indices are odd or all those whose indices are even. In particular, we can choose in two ways which points of each of the $c(g)$ cycles to include in $A$. Hence, there exist $2^{c(g)}$ sets $A$ such that $|A| = n/2$ and $A \cap A^g$ is empty. This gives us
    \[ \sum_{g\in G} \left| \left\lbrace A \in {[n] \choose n/2} \> \middle\vert \> A\cap A^g = \emptyset  \right\rbrace \right| = \sum_{g\in  \mathcal{O}_0} 2^{c(g)}, \]
    proving the lemma in the case when $n$ is even.

    Suppose now that $n$ is odd.
   We start by considering the left-hand side of \cref{eq:itsJustDoubleCounting}.
    Fix
    a set $A\in {[n] \choose \lfloor n/2\rfloor}$.
    Let ${\mathcal B}_A = \{A^g \mid g\in G, A \cap A^g = \emptyset\}$, and for each $B \in {\mathcal B}_A$ choose an element $h_B\in G$ mapping $A$ to $B$. Let $T = \{h_B : B \in {\mathcal B}_A\}$. Clearly, $|T| = |{\mathcal B}_A| \le \lceil n/2 \rceil$.
    Similarly as in the even case, we see
    that for each $B\in {\mathcal B}_A$, the set of elements $g\in G$ mapping $A$ to $B$ equals $G_A h_B$. Now observe that
    $$
    \left| \left\lbrace g \in G \mid A\cap A^g  = \emptyset  \right\rbrace \right|
    =
    \left| \left\lbrace g \in G \mid A^g  \in {\mathcal B}_A  \right\rbrace \right|
    =
    \sum_{B\in {\mathcal B}_A}
    \left| \left\lbrace g \in G \mid A^g = B   \right\rbrace \right|
    =
%    \sum_{h \in T} \left| \left\lbrace g \in G \mid A^g   = A^h  \right\rbrace \right| = 
%    \sum_{h \in T} \left| G_A \right| = 
    |{\mathcal B}_A|\,|G_A|.
    $$
    Therefore
    the left-hand side of \cref{eq:itsJustDoubleCounting}
    is equal to
    $$\sum_{A \in {[n] \choose \lfloor n/2 \rfloor}} |{\mathcal B}_A| |G_A|, $$
    and since $1\le |{\mathcal B_A}| \le  \lceil n/2 \rceil$, it is bounded below and above by the quantities 
    \[ \sum_{A \in {[n] \choose \lfloor n/2\rfloor }} |G_A|\> \hbox{ and } \> %\sum_{A \in {[n] \choose \lfloor n/2\rfloor }} \left| \left\lbrace g \in G \mid A\cap A^g = \emptyset  \right\rbrace \right| \le 
    \left\lceil \frac{n}{2} \right\rceil\sum_{A \in {[n] \choose \lfloor n/2\rfloor }} |G_A|,\]
    respectively.

    It remains to show that the right-hand side of \cref{eq:itsJustDoubleCounting} is equal to $\sum_{g \in \mathcal{O}_1}2^{c(g)-1}\ell(g)$.
    Again, for a given permutation $g\in G$, we count all the possible $\lfloor n/2 \rfloor$-subsets $A$ such that $A$ and $A^g$ are disjoint. The situation is slightly more complex in this case. In the same way as before, we observe that, if $g$ contains more than one odd cycle, such an $A$ does not exist. Hence, we assume that $g$ contains precisely one odd cycle of length $\ell(g)$, and $c(g)-1$ even cycles. Note that $A$ intersects the odd cycle in precisely $\lfloor \ell(g)/2 \rfloor$ points. Moreover, no two consecutive points of the odd cycle are contained in $A$, implying that precisely two consecutive points are excluded from $A$. Once these two consecutive points are determined, the remaining points of the cycle contained in $A$ are uniquely determined (by the requirement that they need to alternate between being contained and not in $A$). Since there are $\ell(g)$ choice for the two consecutive points excluded from $A$, and repeating the same reasoning as before for the points belonging to an even cycle, we obtain
    \[ \sum_{g\in \mathcal{O}_1} \left| \left\lbrace A \in {[n] \choose \lfloor n/2\rfloor} \> \middle\vert \> A\cap A^g = \emptyset  \right\rbrace \right| = \sum_{g\in  \mathcal{O}_1} 2^{c(g)-1}\ell(g) . \]
    This concludes the proof.
\end{proof}

Given two functions $f,g: \NN \to \mathbb{R}^+$, we write $f \sim g$ if the limit of the quotient $f(n)/g(n)$, as $n$ goes to infinity, is $1$.
\begin{lemma}\label{lemma:MarusaRevenge}
    Suppose that $G$ is a permutation group of degree $n$ such that $\m(G) = \lceil (n+1)/2 \rceil$. Then, if $n$ is even,
    \[ |G| \geq \frac{1}{2^{n/2}} {n \choose n/2} \sim \frac{1}{\sqrt{2\pi}} \cdot 2^{\frac{1}{2} (n - \log_2n)} ,\]
    while, if $n$ is odd,
    \[ |G| \geq \frac{1}{2^{\frac{n-1}{2}}n} {n \choose \frac{n-1}{2}} \sim \sqrt{\frac{2}{\pi}} \cdot 2^{\frac{1}{2} (n - 3\log_2 n)} .\]
\end{lemma}
\begin{proof}
    These bounds are obtained by disregarding the exclusion of the elements from $\mathcal{O}_i$, by substituting in the formulae from \cref{lemma:PrimozRevenge} the lower bound $1$ for $|G_A|$ and the upper bound $2^{n/2}$ for $2^{c(g)}$ and $n$ for $\ell(g)$, and by applying Stirling's formula to derive the asymptotic expansion.
\end{proof}

We are ready to characterise which primitive permutation groups
meet the putative upper bound.

\begin{proof}[Proof of \cref{lemma:complementB}]
%Since $G$ is a primitive, we can make use of \cref{lemma:MarusaRevenge} through the extensive literature regarding the order of primitive permutation groups. 
By \cite[Theorem~1.1]{Maroti2002}, a primitive group whose degree exceeds $24$ is either a subgroup of $\Sym(a) \wr \Sym(b)$ with socle isomorphic to $\Alt(a)^b$, where the action of $\Sym(a)$ is on $k$-subsets (for some $1\le k \le a/2)$, the action of $\Sym(b)$ in on $b$ points, and the wreath product is endowed with the product action, or
\begin{equation}\label{eq:Maroti}
    |G| \leq n^{1 + \log_2n}.
\end{equation}
By elementary computations and by using the improved error term for Stirling formula proved in \cite{Robbins1995}, we show that, for even integers $n\le 300$,
\[ \frac{1}{2^{n/2}} {n \choose n/2} \ge n^{1 + \log_2n} ,\]
and that, for odd integers $n\le 300$,
\[\frac{1}{2^{\frac{n-1}{2}}n} {n \choose \frac{n-1}{2}} \ge n^{1 + \log_2n}.\]
Moreover, by checking a finite number of cases with a calculator, we can decrease these bounds even further. Indeed, the function appearing in \cref{lemma:MarusaRevenge} is greater that the bound from \cref{eq:Maroti} whenever $n$ is an even integer greater or equal to $116$, or $n$ is an odd integer greater or equal to $147$. 

We have two cases to consider. If the degree of the permutation group $G$ exceeds $116$ or $147$, depending on the parity of its degree, by \cref{thm:productAction}, for every $a \geq 4$ and $b\geq 2$,
\[ \m(\Sym(a) \wr \Sym(b)) \leq \left\lceil \frac{a+1}{2} \right\rceil^b < \left\lceil \frac{a^b+1}{2} \right\rceil.\]
In particular, if the degree exceeds $16$, we can assume that  $G$ is of almost simple type. Moreover, \cref{thm:ksets} implies that $G$ is permutation isomorphic to either $\Alt(n)$ or $\Sym(n)$ endowed with their actions on $n$ points.

Let us now deal with the groups of small degree, where we can use a computational approach. We started with a list of all primitive permutation groups of degree at most $146$ in the database of primitive groups in {\sc Magma}, where we excluded the symmetric groups and the alternating groups in their natural actions from consideration. (This database is based on \cite{CouttsQuickRD11,RD05,RDunger03,Sims70}.) Each of these primitive groups of even degree was tested against the inequality 
\[{n \choose \frac{n}{2}} \le \sum_{g\in \mathcal{O}_0} 2^{c(g)} ,\]
while those of odd degree against
\[{n \choose\frac{n-1}{2}} \le \sum_{g\in \mathcal{O}_1} 2^{c(g)-1} \ell(g) .\]
Both inequalities are obtained from \cref{lemma:PrimozRevenge}, by substituting to each $|G_A|$ a $1$ in the sum on the right hand side of the expressions.

Our search list get significantly shortened: the possible groups have degree at most $32$, with the only candidate of degree $32$ being ${\rm AGL}(5,2)$. We use a different approach for it. By randomly sampling subsets of size at most $16$ of the domain $\FF_2^5$ and then testing them for existence of a group element that maps them to their complement, we found a number of sets of size $15$ that are not self-separable, thus showing that $\m({\rm AGL}(5,2)) \le 15$ (and we believe that $15$ is the exact value). The remaining groups all have degree $24$ or less, with the only group of order $24$ being the Mathieu group $M_{24}$. These groups were dealt with an exhaustive search of $\lfloor(n+1)/2\rfloor$-subsets which are not self-separable. The groups where such a set was found were excluded form the list, and those that survived are precisely the ones listed in \cref{lemma:complementB}.
Our code for these computations can be found in \cite{ourAlgorithm}. This concludes the proof.
\end{proof}

\subsection{Imprimitive groups meeting the upper bound} 
\label{sub:upperImprimitive}

Let us now move to imprimitive permutation groups $G$ meeting the upper bound on the parameter $\m(G)$ from \cref{thm:upperBound}.
A complete characterisation of
these groups is given in \cref{lemma:caseA,lemma:caseB},
which together with
its primitive counterpart (Lemma~\ref{lemma:complementB}) finishes the proof of \cref{thm:upperBound}.

By \cref{thm:largeBlocks}, the degree $n$ of an imprimitive permutation group $G$ with $\m(G) = \lceil  (n+1)/2 \rceil$
is even, and $G$ admits either a system of $n/2$ blocks of size $2$ or two blocks of size $n/2$. We shall consider these two
cases in separately.
We start our analysis with the following lemma of independent interest, which gives a necessary condition for a permutation group to be $k$-homogeneous.

\begin{lemma}
\label{lem:swing}
    Let $G$ be a transitive permutation group acting on a set $\Omega$ of size $n$, and let $k \le \lceil n/2\rceil - 1$.
    Suppose that, for every two disjoint subsets $A,B \subseteq \Omega$ of size $k$, there exists $g\in G$ mapping $A$ to $B$.
    Then, $G$ is $k$-homogeneous.
\end{lemma}
\begin{proof}
   Aiming for a contradiction, suppose that the statement is false. Then the set
   $$\cE = \{(A,B) \mid A,B \subseteq \Omega, |A| = |B| = k, A^g \neq B  \text{ for every } g \in G\}$$
   is nonempty.  Among all $(A,B) \in \cE$, we choose one which minimises $|A \cap B|$. Observe that, by our assumptions, the size $t$ of the intersection $A\cap B$ satisfies $1\le t \le k-1$. Thus,
   $$
   |\Omega \setminus (A\cup B)|  = n - 2k+t \ge t+1.$$
   This allows us to choose a subset $C\subseteq \Omega \setminus (A\cup B)$ with $|C|=t$. Now, let
   $$A' = (B\setminus A) \cup C ,$$
   and observe that $A\cap A'$ is empty.
   Hence, by the assumptions, there exists $g\in G$ such that $A^g = A'$.
   
   Observe that the sets $A\setminus B$ and
   $\Omega\setminus(A\cup B \cup C)$ are 
   disjoint and of cardinalities
   $k-t$ and $n-2k$, respectively. Their union is, therefore, of size $n-k-t$. If $n-k-t \ge k$, then let $A''$ be an arbitrary $k$-subset of this union. Note that such a set $A''$ is disjoint from $A'$, implying that there exist $g' \in G$ such that
   $A'^{g'} = A''$. Observe also that $B\cap A'' = \emptyset $, implying that there exists $h\in G$ such that $A''^h=B$.
   But then $A^{gg'h} = B$, contradicting the fact that $(A,B) \in \cE$. This contradiction implies that $n-k-t < k$,
or equivalently, $2k+t-n = k -(n-k -t) > 0$.
  Since $2k \le n-1$, it follows that $2k+t-n \le t-1$, allowing us to choose a $(2k-n+t)$-subset $D$ of $A\cap B$. Now let
   \[
   A'' = (A \setminus B) \cup (\Omega \setminus (A\cup B \cup C)) \cup D,
   \]
   and observe that $|A''| = k$ and $A'' \cap A' = \emptyset$.
   Similarly as above, there exist $g' \in G$ such that
   $A'^{g'} = A''$. Observe also that $|B\cap A''| = |D| < t$, implying that there exists $h\in G$ such that $A''^h=B$, which contradicts
   the fact that $(A,B) \in \cE$.

   Hence $\cE$ is empty, and the statement is proven.
   %Note that we can choose such a set $D$ since the sets $A\cap B$ and $(A \setminus B) \cup (\Omega \setminus (A\cup B \cup C))$ have empty intersection, 
   %$|\Omega \setminus (A \cup B \cup C))| \ge 1$ by \cref{eq:swing_t},
   %and
   %\[ |D| \le k - |(A \setminus B) \cup (\Omega \setminus (A\cup B \cup C))| \le k- (k - t + 1) \le t - 1 = |A \cap B| -1. \] This also shows that $|A'' \cap B| < t$.
    %Observe that $A'' \cap A' = \emptyset$, hence, by assumptions, there exists $g'\in G$ such that $A'^{g'} = A''$. By the assumption on the minimality of the intersection of $A$ and $B$ among all pairs such that one set cannot be mapped into the other by an element of $G$, it now follows that there exists $g''$ such that $A''^{g''} = B$, hence $g g'g''$ maps $A$ to $B$, and we have a contradiction.
\end{proof}
\begin{remark}
    Note that the upper bound on $k$ in the assumptions of the above lemma cannot be improved to
    $\lceil n/2 \rceil$
    since that would imply that
    every transitive permutation group $G$ of even degree $n$ and with $\m(G) = n/2$ is $k$-homogeneous. However, as we show later in this section, this is far from true.
\end{remark}

\subsubsection{Permutation groups admitting blocks of size $2$}
We shall now turn our attention to transitive permutation groups $G$ that admit a system of imprimitivity $\Sigma$ consisting of blocks of size $2$. A natural approach to analyse such groups is by considering the permutation module over the field or order $2$ of the permutation group $G^\Sigma$, induced by the action of $G$ on $\Sigma$. While a complete classification of permutation groups with blocks of size $2$ is beyond the scope of this paper, we shall pursue this strategy to the point that is needed to prove
Lemma~\ref{lemma:caseA} and leave a more thorough analysis for the future.

Suppose that $G$ is a transitive permutation group on a set $\Omega$, $|\Omega| =2m$, admitting a system of imprimitivity $\Sigma$ with $m$ blocks of size $2$. Let $H=G^\Sigma$ denote the permutation group induced by the action of $G$ on $\Sigma$. Then $G$ is permutation isomorphic to a subgroup of $C_2 \wr H$ in its imprimitive action on $\{0,1\} \times \Sigma$.
Let us make this more explicit. 

Let $\FF_2^\Sigma$ be the vector space of all functions from $\Sigma$ to $\FF_2$, endowed with the point-wise addition.
Further, for $\chi \in \FF_2^\Sigma$ and $g\in H$, let $\chi^g \in \FF_2^\Sigma$ be defined by $\chi^g(B) = \chi(B^{g^{-1}})$ for every $B \in \Sigma$. This defines a linear representation $H \to {\rm GL}(\FF_2^\Sigma)$ and thus
turns $\FF_2^\Sigma$ into an $\FF_2H$-module, called the {\em permutation module of $H$}.

Further, for $B\in \Sigma$, define $\tau_B$ to be the permutation on $\Omega$ swapping the two elements in  $B$ and fixing everything else. It is convenient to identify $\tau_B$ with the characteristic function $\chi_B:\Sigma \rightarrow \FF_2$ of the element $B\in \Sigma$. Such an identification clearly yields an isomorphism
between the group $N:= \langle \tau_B \mid B \in \Sigma \rangle$ and the additive group $\FF_2^\Sigma$.
Moreover, under this identification, the conjugate $(\tau_B)^g = \tau_{B^g}$ of 
$\tau_B$ by $g\in G$ corresponds to the function $(\chi_B)^{\bar{g}}$, where $\chi_B$ is the characteristic
function of $B$ representing $\tau_B$ and $\bar{g}$ is the permutation induced by $g$ on $\Sigma$.
This shows that we may think of $N$ not only as the vector space $\FF_2^\Sigma$ but also as the $\FF_2H$-permutation
module (with the multiplication by $h\in H \le \FF_2H$ given via conjugation by an element $g\in G$, whose induced action on $\Sigma$ is $h$).

Let us now turn our attention to the kernel $K$ of the action of $G$ on $\Sigma$. Note that $K\le N$, and hence $K = G \cap N$. Since $K$ is normal in $G$, the above identification of $N$ with $\FF_2^\Sigma$ turns $K$ into the $\FF_2H$-submodule of the permutation module $N \cong \FF_2^\Sigma$. Understanding imprimitive groups with blocks of size $2$ is thus reduced to the study of submodules of the permutation modules these groups induce on the block system, together with the question of determining different (not permutation isomorphic) extensions of these modules by the induced permutation groups on the block system.

Note that the $\FF_2H$-permutation module always contains the trivial submodule, the one dimensional {\em trace module} $\mathsf{T}$ spanned by $\Pi_{B \in \Sigma} \tau_B$, and the {\em augmentation module} $\A$ of codimension $1$ consisting of all products of an even number of $\tau_B$'s. While other submodules might exist in general, 
 this will not be the case for most of the groups $H$ appearing in our analysis.

We shall now present a series of lemmas and constructions that determine some transitive permutation groups (up to isomorphism) admitting blocks of size $2$ that appear in our analysis. 
We begin with the case in which $K$ is trivial.

\begin{lemma}
 \label{lem:trivialK}
    Let $G$ be a transitive permutation group on  $\Omega$ admitting a system of imprimitivity $\Sigma$ consisting of $m$ blocks of size $2$. Let $B \in \Sigma$ and let $G_B$ be the setwise stabiliser of $B$ in $G$. 
    Suppose that the kernel of the action of $G$ on $\Sigma$ is trivial.
    Then, there exists an index $2$ subgroup $S$ of $G_B$, with trivial core in $G$, 
    and $G$ is permutation isomorphic to the action of $G$ on the coset space $G/S$ by right multiplication.
\end{lemma}

\begin{proof}
    Let $\omega \in \Omega$ and let $B \in \Sigma$ be such that $\omega \in B$. Take $S$ to be the vertex-stabiliser $G_\omega$ and observe that $S$ has index $2$ in $G_B$ and is core-free in $G$. Furthermore, since $S = G_\omega$, $G$ is permutation isomorphic to the action of $G$ on the coset space $G/S$ by right multiplication.
\end{proof}

\begin{construction}
\label{rem:cos}
In the case when $G= \Sym(m)$ acting on the cosets of the subgroup $S=\Alt(m-1)$,
the permutation action of $G$ on the coset space $G/S$ has a more combinatorial description. 
Consider the set $\Omega:=\{1,\ldots,m\} \times \{0,1\}$. For $g\in G$, let $\bar{g}$ be the permutation of $\Omega$ defined by $(\sigma,i)^{\bar{g}} = (\sigma^g, i)$ for each $(\sigma,i) \in \Omega$,
and let $\tau$ be the permutation acting as $(\sigma,i)^\tau=(\sigma,1-i)$. 
Now consider the permutation group 
$$\mix(\Sym(m),\Alt(m),\one) \> := \> \{\bar{g} : g \in \Alt(m)\}\> \cup\> \{\tau\bar{g} : g \in \Sym(m)\setminus \Alt(m)\}.$$
Observe first that the function $\iota \colon \Sym(m) \to \mix(\Sym(m),\Alt(m),\one)$ mapping $g$ to $\bar{g}$ if
$g\in \Alt(m)$ and to $\tau\bar{g}$ if $g\in \Sym(m)\setminus \Alt(m)$ is an isomorphism of (abstract) groups,
mapping the subgroup $\Alt(m-1)\le \Sym(m)$ to the stabiliser of the point $(m,0) \in \Omega$ in $\mix(\Sym(m),\Alt(m),\one)$. This implies
that $\mix(\Sym(m),\Alt(m),\one)$ is permutation isomorphic to the group $\Sym(m)$ acting on the cosets of $\Alt(m-1)$.
\end{construction}

%\begin{remark} \label{rem:K=1}
Lemma~\ref{lem:trivialK} can also be used in the case where the kernel $K$ of the action of $G$ on $\Sigma$ is non-trivial and possesses
a complement $\overline{H} \cong G^\Sigma$ which is transitive on $\Omega$ (the situation where the complement is intransitive is considered in Lemma~\ref{lem:semidir}). Namely, one can then let $\overline{H}$ take the place of $G$ in the lemma, and conclude that the block-stabiliser $\overline{H}_B$ has a core-free index $2$ subgroup $S$ with $\overline{H}$ acting on $\Omega$ permutation isomorphic to the coset action on $\overline{H}/S$. This will be evident in \cref{constr:mix}. 
%\end{remark}

We now consider the case in which $K$ is nontrivial and it admits an intransitive complement in $G$.

\begin{construction} \label{constr:semidirect}
    Let $H$ be a transitive permutation group with domain $\Sigma$, Let $K$ be a nontrivial submodule of the permutation $\FF_2H$-module $\FF_2^\Sigma$, and let $\Omega = \Sigma \times \{0,1\}$. For each $h \in H$ let $\overline{h}$ be the permutation on $\Omega$ defined by $(\sigma, i) \mapsto (\sigma^h, i)$ and let $\overline{H} = \{\overline{h} \mid h \in H\}$. Further, for each $\chi \in K$, let $\overline{\chi}$ be the permutation of $\Omega$ which swaps the two elements in $\{ \sigma \} \times \{0,1\}$ when $\chi(\sigma) = 1$, and fixes them when $\chi(\sigma) = 0$. Finally, let $\overline{K} = \{\overline{\chi} \mid \chi \in K\}$. Observe that $K \cong \overline{K}$, $H \cong \overline{H}$, $\overline{K} \cap \overline{H} = 1$, and that $\overline{K}$ normalises $\overline{H}$.
    This allows us to define the permutation group $\overline{K} \rtimes \overline{H}$ on $\Omega$, which we shall denote simply as $K \rtimes H$, and when $K=\mathsf{T}\cong C_2$ is the trace submodule,
    we shall also write $C_2\times H$ instead. Note that, since $K$ is nontrivial, $K \rtimes H$ is transitive, with a system of imprimitivity $\overline{\Sigma} = \{\{\sigma\} \times \{0,1\} \mid \sigma \in \Sigma\}$. Moreover, the kernel of the action of $G$ on $\overline{\Sigma}$ is $\overline{K} \cong K$ and the induced permutation group on $\overline{\Sigma}$ is permutation isomorphic to $H$.
\end{construction} 

\begin{lemma}
     \label{lem:semidir}
     Let $G$ be a transitive permutation group on $\Omega$ admitting a system of imprimitivity $\Sigma$ consisting of $m$ blocks of size $2$, let $H$ be the permutation group induced by $G$ on $\Sigma$ and let $K$ be the kernel of the action of $G$ on $\Sigma$ (viewed as a submodule of the permutation module $\FF_2^\Sigma$). Suppose that $K$ is nontrivial and it has a complement in $G$ which is intransitive on $\Omega$. Then $G$ is isomorphic to $H \rtimes K$, as defined in \cref{constr:semidirect}.
\end{lemma}

\begin{proof}
    Let $\overline{H}$ be an intransitive complement of $K$ in $G$. Since $K$ acts trivially on $\Sigma$, $\overline{H}$ acts transitively, and since it is intransitive on $\Omega$, it intersect each block in $\Sigma$ in exactly one point. Without loss of generality we can label the points in a block $\sigma \in \Sigma$ by $(\sigma ,0)$ and $(\sigma, 1)$ in such a way that the two orbits of $H$ are $\Sigma \times \{0\}$ and $\Sigma \times \{1\}$. It is straightforward to verify that under this labelling, $\overline{H}$ is precisely the corresponding permutation group $\overline{H}$ from \cref{constr:semidirect}, and that $K$ acts on $\Sigma$ in the same way as $\overline{K}$. In particular, this labelling yields a permutation isomorphism between $G$ and $H \rtimes K$.
\end{proof}

The following construction and lemma deal with the situation where the kernel $K$ is the augmentation ideal.

\begin{construction} \label{constr:mix}
     Let $H$ be a transitive permutation group on the domain $\Sigma$, let $S$ be an index $2$ normal subgroup of $H$, and let $\A$ be the augmentation submodule of $\FF_2^\Sigma$. Consider the group $\overline{\A} \rtimes \overline{H}$, acting on $\Omega = \Sigma \times \{0,1\}$, as explained in \cref{constr:semidirect}, and observe that it contains $\overline{\A}  \rtimes \overline{S}$ as an index $2$ subgroup. Choose $\tau$ in $\FF_2^\Sigma \setminus \A$ and $h \in H \setminus{S}$, and let $G =\langle  \overline{\A}  \rtimes \overline{S}, \tau \overline{h} \rangle$.
     Observe that $\overline{\A}  \rtimes \overline{S}$ has index $2$ in $G$, that $G$ does not depend on the choice of $\tau$ and $h$, and that it is a transitive permutation group on $\Omega$ with a system of imprimitivity 
     $\overline{\Sigma} = \{\{\sigma\} \times \{0,1\} \mid \sigma \in \Sigma\}$. Moreover, the kernel of the action of 
     $G$ on $\overline{\Sigma}$ is $\overline{A}$ and
     $G^{\overline{\Sigma}}$ is permutation isomorphic to $H$. We shall denote the group $G$ as $\mix(H,S,\A)$.
     \end{construction}

\begin{lemma}
 \label{lem:A}
    Let $G$ be a transitive permutation group on $\Omega$ admitting a system of imprimitivity $\Sigma$ consisting of $m$ blocks of size $2$. Suppose that the kernel of the action of $G$ on $\Sigma$ is the augmentation ideal $\A$. Let $H=G^\Sigma$ be the permutation group induced by the action of $G$ on $\Sigma$. 
    Then, either $G$ is isomorphic to $\A \rtimes H$, or $H$ has an index $2$ subgroup $S$ and $G$ is isomorphic to $\mix(H,S,\A)$, as defined in \cref{constr:mix}.
\end{lemma}

\begin{proof} 
For each $B \in \Sigma$, choose one point from $B$ and label it by $(B,0)$, and label the other point with $(B,1)$.
This labelling establishes a bijection between $\Omega$ and $ \Sigma \times \{0,1\}$, allowing us
to assume that $\Omega = \Sigma \times \{0,1\}$, with the block system consisting of \emph{fibres} $\{\sigma\} \times \{0,1\}$,
for $\sigma\in \Sigma$.
Let $\overline{H}$ and $\overline{\A}$ be as in \cref{constr:mix}. Observe first that $G$ is an index $2$ subgroup of $\FF_2^\Sigma \rtimes \overline{H}$. Observe that $\overline{\A} \rtimes \overline{H}$ has the prescribed properties, and hence $G=\overline{\A} \rtimes \overline{H}$ gives the first possibility of the statement.
 
We can thus assume that $X := G \cap (\overline{\A} \rtimes \overline{H})$ is a normal subgroup of $\FF_2^\Sigma \rtimes \overline{H}$ of index $4$.
 Since $G$ contains $\A$, we also see that $X = \overline{\A}\rtimes (G \cap \overline{H})$. 
 This implies that $\overline{S} := G \cap \overline{H}$ has index $2$ in $G$.
 Let $S = \overline{S}^\Sigma$, that is, the subgroup of $H$ containing all those $h\in H$ for which $\overline{h} \in \overline{S}$.
 Then clearly $S$ has index $2$ in $H$, and we may apply Construction~\ref{constr:mix} to obtain the group $\mix(H,S,\A)$.
 
 Since $X$ is a normal subgroup of index $4$ in $\FF_2^\Sigma \rtimes \overline{H}$,
the open interval of subgroups between $X$ and $\FF_2^\Sigma \rtimes \overline{H}$ consists of either one or three groups
(depending on whether the quotient by $X$ is cyclic or the Klein group). However, we see that this interval contains
three distinct groups, namely
 $\overline{\A} \rtimes \overline{H}$,
 $\FF_2^\Sigma \rtimes S$, 
 and $\mix(H,S,\A)$. It is straightforward to check that precisely $\A \rtimes H$ and $\mix(H,S,\A)$ are such that the group $G^\Sigma$ is $H$ and that $K = \A$.
 Hence, $G$ is permutation isomorphic to one of these two groups.
 \end{proof}
 
\begin{lemma}
\label{lemma:caseA}
    Let $m \ge 2$ 
    be a positive integer, and let $G$ be an imprimitive group with block system $\Sigma$ that consists of $m$ blocks of size $2$. 
    Then $\m(G) = m+1$ if and only if $G$ is permutation isomorphic to one of the following:
    \begin{enumerate}[$(i)$]
        \item $C_2\times C_2$, $C_4$ or $\Dih(4)$ acting transitively on $4$ points;
        \item $\mix(\Sym(m),\Alt(m),\one), $where $m\ge 3$ (see Remark~\ref{rem:cos});
        \item $\mathsf{T}\rtimes H \cong  C_2 \times H$ (see Construction~\ref{constr:semidirect}), where $m$ and $H$ satisfy one of:
          \begin{itemize}
             \item $m\ge 3$ and $H=\Sym(m)$;
             \item $m\ge 4$ and $H=\Alt(m)$; 
             \item $m=4$ and $H=C_2\times C_2$ or $\Dih(4)$;
             \item $m=6$ and $H=\PGL_2(5)$.
          \end{itemize} 
        \item $\A\rtimes H \cong  C_2^{m-1} \rtimes H$ (see Construction~\ref{constr:semidirect}), where $m$ and $H$ satisfy one of:
          \begin{itemize}
             \item $m\ge 4$ is even and $H=\Alt(m)$ or $\Sym(m)$; 
             \item $m=4$ and $H=C_2\times C_2$  or $\Dih(4)$;
             \item $m=6$ and $H=\PGL_2(5)$;
             \item $m=8$ and $H=\PSL_2(7)$, $\PGL_2(7)$, or $\ASL_3(2)$;
             \item $m=12$ and $H=M_{12}$. 
          \end{itemize}
        \item $\mix(H,S,\A)$ (see Construction~\ref{constr:mix}), where $m$, $H$ and $S$ are one of:
          \begin{itemize}
             \item $m\ge 3$, $H=\Sym(m)$ and $S=\Alt(m)$;
             \item $m=4$ and $(H,S)=(C_2\times C_2,C_2)$, $(\Dih(4),C_2\times C_2)$, or $(\Dih(4),C_4)$;
             \item $m=5$, $H=\AGL_1(5)$ and $S=\Dih(5)$;                        
             \item $m=p+1$, $H=\PGL_2(p)$ and $S=\PSL_2(p)$, for $p\in\{5,7\}$.
          \end{itemize}
        \item $C_2 \wr H \cong C_2^m \rtimes H$ where $m$ and $H$ satisfy one of:
        \begin{itemize}
           \item $m\ge 3$ and $ H = \Sym(m)$;
           \item $m\ge 4$ and $ H = \Alt(m)$;
           \item $m=4$ and $H = C_2\times C_2$ or $\Dih(4)$;
           \item $m=5$ and $H=\AGL_1(5)$;
           \item $m=6$ and $H=\PGL_2(5)$;
           \item $m=8$ and $H=\PSL_2(7)$, $\PGL_2(7)$, or $\ASL_3(2)$;
           \item $m=9$ and $H=\PSL_2(8)$ or $\PGammaL_2(8)$;
           \item $m=12$ and $H=M_{12}$.
        \end{itemize}
       \item one of six transitive groups of degree $8$ or $16$ that
       can be found in the Library of Transitive Groups in \cite{magma} under the names {\tt TransitiveGroup}$(n,i)$ for $n=8$ and $i \in \{17,18,19\}$ and for
       $n=16$ and $i \in \{1506, 1804, 1805\}$.
    \end{enumerate}
\end{lemma} 

\begin{remark}
   Let us provide some additional information on the six exceptional groups appearing in the last item of the lemma. For each of this groups $G$, let
       $H:=G^\Sigma$ and let $K$ be the kernel of the action of $G$ on $\Sigma$,
    viewed as a submodule of the permutation $\FF_2H$-module $\FF_2^\Sigma$.

     The three exceptional transitive groups of degree $8$ all arise from the group $H = \Dih(4)$ (in its action on 4 points) and the $2$-dimensional submodule $K \le \FF_2^\Sigma$ generated the characteristic function of a block of imprimitivity of $\Dih(4)$. The group {\tt TransitiveGroup}$(8,18)$ is then just the semidirect product $K\rtimes H \cong \FF_2^2\rtimes \Dih(4)$ (see Construction~\ref{constr:semidirect}). The other two groups of degree $8$ are both non-split extensions of $K$ by $\Dih(4)$ (that is, in both cases, $K$ has no complement in $G$). Interestingly, both of them
     contain two non-conjugate copies of $\Dih(4)$, one acting regularly on the 8 points and one intransitively, but both intersecting $K$ non-trivially. Moreover,
     all three groups {\tt TransitiveGroup}$(8,i)$, $i\in\{17,18,19\}$, also admit a system of imprimitivity with two blocks of size $4$, meaning that they will appear again later, when this situation is analysed.

     Further, two of the three exceptional groups of degree $16$ induce the group $H=\ASL_3(2)$ (in its natural affine action) with the kernel $K$ corresponding to the unique $4$-dimension submodule of the permutation $\FF_2H$-module $\FF_2^\Sigma \cong \FF_2^8$.
     One of these two groups ({\tt TransitiveGroup}$(16,1804)$, to be precise) is a semidirect product $K\rtimes H \cong \FF_2^4\rtimes\ASL_3(2)$, while {\tt TransitiveGroup}$(16,1805)$ is a non-split extension of $K$ by $\ASL_3(2)$ (and contains no subgroups isomorphic to $\ASL_3(2)$).

     Finally, the remaining exceptional group, namely {\tt TransitiveGroup}$(16,1506)$
     is simply a semidirect product $K\rtimes H$
     where $H=\PSL_2(7)$ in its natural action on $8$ points, and where $K\cong \FF_2^4$ is a cyclic $\FF_2\PSL_2(7)$-module generated by the characteristic function of the set $\{x^2 \mid x \in \FF_7\setminus\{0\}\} \cup \{\infty\}$, viewed as the subset of the projective line $\FF_7 \cup \{\infty\}$.  
\end{remark}

\begin{proof}[Proof of \cref{lemma:caseA}]
 For $m\le 9$, the validity of the lemma has been checked computationally by checking all transitive groups of degree at most $18$ from the Library of Transitive Groups in Magma \cite{magma} (the results and the code are available at \cite{ourAlgorithm}). We shall  assume henceforth that $m\ge 10$, that is, that the degree of $G$ is at least $20$.

\smallskip
\noindent \textsc{Let us show that all the permutation groups $G$ of degree $2m\ge 20$
listed in \cref{lemma:caseA} satisfy
$\m(G) = m+1$.}
The permutation groups $G$ that we need to consider are:
\begin{itemize}
    \item $\mix(\Sym(m),\Alt(m),\one)$;
    \item $C_2 \times \Alt(m)$; and
    \item $\A \rtimes M_{12}$.
\end{itemize}
Namely, all other groups listed in the lemma are overgroups of the above.

Suppose first that $G$ is either $\mix(\Sym(m),\Alt(m),\one)$
or $C_2\times \Alt(m)$. Then $G$ acts 
on the set $\Omega=\{1,\ldots,m\}\times \{0,1
\}$, as explained in Remark~\ref{rem:cos} or Construction~\ref{constr:semidirect}.
In both cases, for every $g\in \Alt(m)$, 
the group $G$ contains the permutation
$\bar{g}$ acting on $\Omega$ according to the rule $(x,i)^{\bar{g}} = (x^g,i)$. Moreover,
if $G=C_2\times \Alt(m)$, then $G$ contains
also the central permutation $\tau$, defined by
$(x,i)^\tau = (x,1-i)$, and if $G=\mix(\Sym(m),\Alt(m),\one)$, then
$G$ contains an element $\tau\bar{g}$ for every
$g\in \Sym(m)\setminus \Alt(m)$.

For $i\in \{1,\ldots,m\}$ let $\sigma_i = \{(i,0),(i,1)\} \subseteq \Omega$.
Further, let $A$ be a subset of $\Omega$ of size $m$, and let $X,Y,Z$ be subsets of $\{1,\ldots,m\}$ defined by
\[X =\{i \mid  \sigma_i \subseteq A\} , \quad Y =\{i \mid \sigma_i \cap A =\emptyset\} , \quad \hbox{and} \quad   Z =\{i \mid |\sigma_i \cap A| =1\} \,.\]
Note that the sets $X,Y,Z$ form a partition of
$\{1,\ldots,m\}$ and that $2|X|+Z = |A|$.
Since $|A| = m = |X|+|Y|+|Z|$,
it follows that $|X| = |Y|$.
Furthermore, for $j\in \{0,1\}$,
let $Z_j = \{i \in Z \mid (i,j) \in A\}$
and observe that $Z_0$ and $Z_1$ form
a partition of $Z$.

If $|X| \ge 1$, then there exists an odd permutation $g\in \Sym(m)\setminus \Alt(m)$
that fixes every element in $Z$ pointwise
and maps $X$ to $Y$ (note that such a permutation always exists).
In this case, the permutation $\tau\bar{g} \in \mix(\Sym(m),\Alt(m),\one)$,
maps $A$ to its complement, as required.

Moreover, if $|X|\ge 2$, then there is also
an even permutation $h\in \Alt(m)$ that
fixes $Z$ pointwise and swaps $X$ with $Y$.
This yields an element $\tau\bar{h}\in C_2\times \Alt(m)$ mapping $A$ to its complement, as required. Similarly, if $|X| = 1$,
then $|Z| = m-2 \ge 8$ allowing us to pick
two distinct elements $s,t \in Z_j$ for some $j\in \{0,1\}$.
Now, let $g$ be the transposition $(s\, t)\in \Sym(m)$ and let $h \in \Sym(m)$ be the transposition swapping $X$ and $Y$.
Then $\tau \bar{g}\bar{h}$ is an element of $C_2\times \Alt(m)$ that maps $A$ to its complement.

Finally, if 
 $X = Y = \emptyset$, then $|Z| = m \ge 3$ (we are in fact assuming $m\ge 10$),
implying that there exists $j\in \{0,1\}$ and $s,t \in Z_j$, $s\not = t$. Let $g=(i\, j) \in \Sym(m)$. If $G=C_2 \times \Alt(m)$,
then $\tau$ is an elements of $G$ that maps
$A$ to its complement. On the other hand,
if $G=\mix(\Sym(m),\Alt(m),\one)$, then
$\tau\bar{g} \in G$ will do the same.

We are now left with the case where
$G=\A \rtimes M_{12}$, acting on the set $\Omega = \{1,\ldots,12\} \times \{0,1\}$.
The proof that
every set of size $12$ is self-separable
in this case was done computationally,
but with the following short-cut that significantly
sped up the computations.
Observe first that the group $M_{12} \le \Sym(12)$ is $5$-transitive. Moreover,
that for every integer $k\le 6$ and every two
disjoint subsets $X, Y \subseteq \{1,\ldots,12\}$ there exists $g\in M_{12}$
simultaneously mapping $X$ to $Y$ and $Y$ to $X$.
Note that for $k=6$ this fact is equivalent
to saying that $M_{12}$ (in its action on $12$ points) meets the upper bound $\m(M_{12}) = 6$,
which was already established in Lemma~\ref{lemma:complementB},
while for $k\le 5$, due to $5$-transitivity of $M_{12}$, we may assume that $X=\{1,\ldots,k\}$
and that $Y$ contains elements $k+1$, $\dots$, $k+5$.

Once this is established (computationally or theoretically),
we may take a subset $A\subseteq \Omega$ of size $12$
and define the sets $\sigma_i$, $X$, $Y$ and $Z$, as well as the permutations $\tau$ and $\bar{g}$ for $g\in M_{12}$, in the same way as we did in the case of the groups $C_2\times \Alt(m)$. 
Let $g\in M_{12}$ be the permutation, existence
of which we have established computationally, that swaps the sets $X$ and $Y$.
Since $|Z| = 12 - |X| - |Y| = 12 - 2|X|$
is even in this case, there exists an element
$\tau_Z \in \A$ that swaps the two elements in $\sigma_i$
for each $i\in Z$ and fixes all other elements of $\Omega$. Then $\tau_Z\bar{g}$ is an element
of $G$ which maps $A$ to its complement.
This completes the proof that all the groups $G$ of degree $2m$
listed in Lemma~\ref{lem:A} satisfy $\m(G)=m+1$.

\smallskip
\noindent \textsc{Let us now prove that, if $\m(G) = m+1$ and $2m\ge 20$, then $G$ appears in the list in \cref{lemma:caseA}.}
 Let $K$ denote the kernel of the action of $G$ on $\Sigma$ and let $H = G^\Sigma \cong G/K$ denote the permutation group induced by this action. As explained earlier in this section, we may view $K$ as a submodule of the permutation module $\FF_2^\Sigma \cong \FF_2^m$.

We shall first prove that the permutation group $H$ (of degree $m$) is $k$-homogeneous, where $k=(\lceil m/2 \rceil - 1)\ge 4$.
By \cref{lem:swing}, it suffices to show that for any pair $X, Y$ of disjoint subsets of the block system $\Sigma$ of size $k$ there exists $g \in G$ such that $X^g =Y$.
Let $X,Y \subseteq \Sigma$ be two such sets, and
let $A$ be a set that contains both points from each block in $Y$, no points from any block in $X$ and precisely one point from each block in $\Sigma \setminus (X \cup Y)$. Note that $|A| = m < \m(G)$, implying that $A$ is self-separable.
Hence there exists $g\in G$ which maps $A$ to its complement and thus (viewed as a permutation of $\Sigma$) maps $X$ to $Y$. Therefore, $H$ is indeed $k$-homogeneous, as claimed.
The list of $k$-homogeneous groups of degree $m$ with $k = (\lceil m/2 \rceil - 1)$ is quite short: by combining  \cite[Theorem~9.4B]{DixonMortimer} and \cite[Section~7]{Cameron1981} (and using the assumption that the degree of $G^\Sigma$ is at least $10$), we see that 

\begin{enumerate}[$(a)$]
    \item $\Alt(m) \le H \le \Sym(m)$, or
    \item $m = 12$ and $H \cong M_{12}$.
\end{enumerate}

As was proved in \cite{Mortimer1980}, the only submodules of the permutational module $\FF_2^\Sigma$ of $H$
are then either the trivial module $\mathbf{1}$, the trace module $\mathsf{T}$, the augmentation module $\A$, and of course
$\FF_2^\Sigma$ itself.

Suppose first that the kernel $K$ is trivial. By Lemma~\ref{lem:trivialK}, the stabiliser of a point in $H$ contains
a subgroup $S$ of index $2$, excluding the cases $H\cong \Alt(m)$ and $H\cong M_{12}$. Hence $H=\Sym(m)$
and thus $S=\Alt(m-1)$. Lemma~\ref{lem:trivialK} then implies that $G$ is permutation isomorphic to the action of
$\Sym(m)$ on the cosets of $\Alt(m-1)$, and thus permutation isomorphic to $\mix(\Sym(m),\Alt(m),\one)$,
as explained in Remark~\ref{rem:cos}.

Suppose now that $K=\mathsf{T}$. Then $G$ is a central $2$-extension of $H$. Note also that $K=\langle \tau\rangle$ where $\tau$ swaps the two points in every block in $\Sigma$.

If the extension is split with the complement $\overline{H} \cong H$, then $G = C_2 \times \overline{H}$. Further, if $\overline{H}$ is intransitive, then by Lemma~\ref{lem:semidir}, $G = \mathsf{T} \rtimes H$, as in Construction~\ref{constr:semidirect}. A direct computation in the case of $H=M_{12}$
exhibits a non-self-separable of size $12$ (see \cite{ourAlgorithm} for such a set), showing that $\m(\mathsf{T} \rtimes M_{12}) \le 12$,
contradicting our assumptions.

On the other hand, if $\overline{H}$ is transitive, then we may apply the previous paragraph with 
$\overline{H}$ in place of
$G$ to conclude that $\overline{H} \cong \mix(\Sym(m),\Alt(m),\one)$ in its action on $\Omega':=\{1,\ldots,m\} \times \{0,1\}$. In this setting, the block system $\Sigma$ corresponds to the partition of $\Omega'$ into pairs $\{\sigma\} \times \{0,1\}$, and the generator $\tau$ of the kernel $K$ to the permutation
$\tau'$ of $\Omega'$ that swaps both elements in each of these pairs. The permutation group $G$ is thus permutation
isomorphic to the action of $\langle \overline{H},\tau'\rangle$ on $\Omega'$, which is clearly isomorphic to the
permutation group $\mathsf{T} \rtimes H$, as in Construction~\ref{constr:semidirect}. This completes the proof
in the case the extension of $K$ by $H$ is split.

Now suppose that $G$ is a non-split extension of $K=\langle \tau \rangle$ by $H$.
Let $\pi \colon G \to G^\Sigma = H$ be the corresponding epimorphism with kernel $K$.
Since the stabiliser $G_\omega$ of a point $\omega \in \Omega$ intersects $K$ trivially,
$\pi$ maps $G_\omega$ isomorphically to the stabiliser
$H_B$ of a point $B\in \Sigma$. Furthermore,
$\pi^{-1}(H_B) = G_B=\langle \tau \rangle \times G_\omega$.

Suppose first that $H=\Sym(m)$. Then (as is well known) there exist
two non-isomorphic non-split central  $2$-extensions, usually denoted by $2.\Sym(m)^+$ and $2.\Sym(m)^-$.
In both cases the central element is contained in the derived subgroup $[P,P]$ of the
the preimage $P$ of the group $\Sym(m-1)$. In our context that implies that $\tau \in [G_B,G_B]$,
which clearly contradicts the fact that $\tau$ is a direct (central) factor in $G_B$.
Similarly, if $H=\Alt(m)$, then $G$ is isomorphic to the unique non-split central  $2$-extension
of $\Alt(m)$, usually denoted by $2.\Alt(m)$. Here it can also be seen that $\tau$ resides within
the derived subgroup $[G_B,G_B]$, again yielding a contradiction.

We are thus left with the case where $H=M_{12}$. Since $M_{12}$ is a perfect simple group and has
the Schur multiplier of order $2$, there exists a unique non-split $2$-extension, namely the group
usually denoted by $2.M_{12}$. It can be checked computationally that this group has a unique
permutation representation of degree $24$ (admitting blocks of size $2$ and acting on them as $M_{12}$),
which can be found in the Database of Transitive Permutation Groups in {\sc Magma} \cite{magma} under the name
{\tt TransitiveGroup}$(24,18440)$. One can check computationally that this group admits a non-self-separable
set of size $12$, implying that $\m(G) \le 12$, contradicting our assumption on $\m(G)$ (see \cite{ourAlgorithm} for the computational details).
%\primozcomment{Marusa, can you find this set explicitly and add that into the file with computations?}
This finishes the proof in the case where $K=\mathsf{T}$.

Let us now assume that $K=\A$. If $H\cong \Alt(m)$ or $M_{12}$, then Lemma~\ref{lem:A} implies that
$G=\A \rtimes H$. 
However, if $m$ is odd, then the subset
$A:= \{1,\ldots,m\} \times \{0\}$
of the domain of $\A \rtimes H$ (see Construction~\ref{constr:semidirect}) cannot be separated from itself since $G$ does not contain
the trace submodule $\mathsf{T}$.
On the other hand, if $H\cong \Sym(m)$, then $G$ is isomorphic either to $\A \rtimes \Sym(m)$ (and as above, $m$ needs to be even)
or to $\mix(\Sym(m),\Alt(m),\A)$.

Finally, if $K=\FF_2^\Sigma$, then the single option is $G=\FF_2^\Sigma \rtimes H = C_2 \wr H$.
This concludes the proof of the lemma.
\end{proof}

This concludes our analysis of permutation groups admitting blocks of size $2$.

\subsubsection{Permutation groups admitting two blocks of size $m$.}
We start by describing some quirky permutation subgroups of $H \wr C_2$.
\begin{construction}\label{con:nonEq}
    Let $H$ be a transitive permutation group of degree $n$, and let $f\in \Sym(n)$ be a permutation such that $f^2 \in H$ (and thus $f \in \mathbf{N}_{\Sym(n)}(H)$). We denote the permutation group
    \[\left\{ (h,h) \mid h\in H \right\} \cong H ,\]
    whose action is diagonal on $\Omega^{\{0,1\}}$,  by $\qrk_0(H)$.
    Note that $\qrk_0(H)$ is intransitive with two orbits, and that its transitive constituents are permutation isomorphic to $H$. % (see \cite[Exercise~1.6.1]{DixonMortimer}).
    We define
    \[\qrk(H,f) = \langle \qrk_0(H), \bar f \rangle ,\]
    where
    \[ (\omega,0)^{\bar f} = (\omega^f, 1), \quad \hbox{and} \quad  (\omega,1)^{\bar f} = (\omega^{f},0) .\]
    Note that $\qrk(H,f)$ is transitive on $\Omega ^ {\{0,1\}}$, and that $\qrk_0(H)$ is an index $2$ subgroup of $\qrk(H,f)$. Further, observe that $\qrk(H,1) = H \times C_2$.
    %To compute the number of permutational isomorphism classes of groups of the form $\qrk(H,f)$, we can assume that the first factor $H$ and its intersection with a stabiliser are fixed, that is, we can conjugate by an element from the normaliser of the intersection of a stabiliser with $\Sym(\Omega \times \{0\})$. As this normaliser contains $\Sym(\Omega \times \{1\})$, the number of permutational isomorphism classes equals the number of permutation isomorphic but non-equivalent actions of $H$ (see \cite[Lemma~1.6B]{DixonMortimer}).
\end{construction}

\cref{con:nonEq} captures, up to permutational isomorphism, every permutation group which is an imprimitive extension of $H$ of degree $2$, as stated in the following lemma.
\begin{lemma}\label{lemma:GodWhy}
    Let $H$ be a permutation group on $\Omega$, let $s\in \Sym(\Omega^{\{0,1\}})$ be a permutation with the properties that it swaps the two sets $\Omega^{\{0\}}$ and $\Omega^{\{1\}}$, and that $s^2 \in \qrk_0(H)$ endowed with its diagonal action on $\Omega^{\{0,1\}}$. Denote by $G$ the permutation group on $\Omega^{\{0,1\}}$ generated by $\qrk_0(H)$ and $s$. Then $G$ is permutation isomorphic to $\qrk(H,f)$, for some $f \in \Sym(\Omega)$.
\end{lemma}
\begin{proof}
    Aiming for a contradiction, suppose that, for every $f \in \Sym(\Omega)$,
    \[ s \notin {\bar f}^ { {\bf N}_{\Sym(\Omega^{\{0,1\}}} \left( \qrk_0(H) \right) }.\]
    In particular, there exist two distinct permutations $a,b \in \Sym(\Omega)$ such that
    \[ (\omega,0)^s = (\omega^a,1), \quad \hbox{and} \quad (\omega,1)^s = (\omega^b,0) .\]
    From the fact that $s^2 \in \qrk_0(H)$, it follows that
    \[ s \in {\bf N}_{\Sym(\Omega^{\{0,1\}})} \left( \qrk_0(H) \right) = \langle \qrk_0\left( {\bf N}_{\Sym(\Omega)} \left( H \right) \right), \bar 1 \rangle .\]
    But no element in the normaliser has the desired form, which is a contradiction. Hence, $s$ and $\bar f$ are conjugate in ${\bf N}_{\Sym(\Omega^{\{0,1\}})} \left( \qrk_0(H) \right)$, which proves the permutational isomorphism (see \cite[Exercise 1.6.1]{DixonMortimer}).
\end{proof}

\begin{example}\label{ex:cringe}
    We specialize \cref{con:nonEq} for the alternating and symmetric groups.
    Let $m\ge 5$ be an integer. Observe that, if $f \in \Alt(m)$ is an even permutation, then $(f,f)\bar f = \bar 1$, and hence
    \[ \qrk(\Alt(m),f) = 
    \qrk(\Alt(m),1) = \Alt(m) \times C_2 .\]
    Analogously, if $f\in \Sym(m)$ is odd, then, by choosing $g\in \Alt(m)$ such that $gf = (1,2)$, we have $(g,g) \bar f = {\overline{(1,2)}}$. Therefore,
    \[ \qrk(\Alt(m),f) = 
    \qrk(\Alt(m), (1,2)) .\]
    On the other hand, the same reasoning shows that, for every permutation  $f \in \Sym(m)$,
    \[ \qrk(\Sym(m),f) = 
    \qrk(\Sym(m), 1) .\]
\end{example}

The appearance of such subgroups in our setting is justified by the following observation.
\begin{lemma}\label{lem:imprimitiveAS}
    Let $H$ be an almost simple primitive permutation group with socle $T$, and let $G$ be an imprimitive permutation group $G$ with two blocks of imprimitivity, $B$ and $C$, such that $G_B^B = H$. Then, either $T \wr C_2 \le G \le H \wr C_2$, or $G$ is permutation isomorphic to $\qrk(H,f)$, for some $f\in \Sym(\Omega)$ as defined in \cref{con:nonEq}.
\end{lemma}
\begin{proof}
    We start by computing the socle of $G$. We note that $G_{(\Sigma)} = G_B = G_C$ and that $G$ is an extension of $C_2$ by $G_{(\Sigma)}$. Let $N$ be a minimal normal subgroup of $G_{(\Sigma)}$. Consider the projections $\pi_B: G_{(\Sigma)} \to G_B^B$ and $\pi_C: G_{(\Sigma)} \to G_C^C$. Since $N$ is minimal normal, so are $N^{\pi_B}$ and $N^{\pi_C}$. Observe that, as $G$ is faithful, the two images cannot be concurrently trivial. In particular, $T$ is a subgroup of $N$, and hence the socle of $G$ is of the form $T^\ell$, for some positive integer $\ell$. As $G \le \Sym(|B|) \wr C_2$, we see that $\ell \le 2$, and hence either $\ell=1$ or $\ell=2$. On one hand, if $\ell=2$, it follows at once that
    \[ T \wr C_2 \le G \le H \wr C_2 .\]
    On the other hand, suppose that $\ell=1$. Since $G$ acts transitively on $\Sigma$, there exists a permutation $s\in G$ such that $B^s = C$ and $s^2 \in G_{(\Sigma)}$. Observe that $G_B^B$ and $G_C^C$ are permutation isomorphic: the requested bijections are $s$ for the domains and the conjugation by $s$ for the groups. By \cref{lemma:GodWhy}, $G$ is permutation isomorphic to $\qrk(H,f)$, where $f$ is a permutation with the property that $\bar f = s$.
\end{proof}

We can now state the main result of the section.
\begin{lemma}\label{lemma:caseB}
    Let $m \ge 2$ 
    be a positive integer, and let $G$ be an imprimitive group with block system $\Sigma$ that consists of two blocks of size $m$. 
    Then $\m(G) = m+1$ if and only if $G$ is permutation isomorphic to one of the following:
    \begin{enumerate}[$(i)$]
        \item $C_2\times C_2$, $C_4$ or $\Dih(4)$ acting transitively on $4$ points;
        %\item $\mix(\Sym(m),\Alt(m),\one), $where $m\ge 3$ (see Remark~\ref{rem:cos});
        \item $\qrk(H,1) \cong H \times C_2$ (see Construction~\ref{con:nonEq}), where $m$ and $H$ satisfy one of:
          \begin{itemize}
             \item $m\ge 3$ and $H=\Sym(m)$;
             \item $m\ge 4$ and $H=\Alt(m)$; 
             \item $m=4$ and $H=C_2\times C_2$ or $\Dih(4)$;
             \item $m=6$ and $H=\PGL_2(5)$;
          \end{itemize} 
        %\item $\A\rtimes H \cong  C_2^{m-1} \rtimes H$ (see Construction~\ref{constr:semidirect}), where $m$ and $H$ satisfy one of:
        %  \begin{itemize}
        %     \item $m\ge 4$ is even and $H=\Alt(m)$ or $\Sym(m)$; 
        %     \item $m=4$ and $H=C_2\times C_2$  or $\Dih(4)$;
        %     \item $m=6$ and $H=\PGL_2(5)$;
        %     \item $m=8$ and $H=\PSL_2(7)$, $\PGL_2(7)$, or $\ASL_3(2)$;
        %     \item $m=12$ and $H=M_{12}$. 
        %  \end{itemize}
        %\item $\mix(H,S,\A)$ (see Construction~\ref{constr:mix}), where $m$, $H$ and $S$ are one of:
        %  \begin{itemize}
        %     \item $m\ge 3$, $H=\Sym(m)$ and $S=\Alt(m)$;
        %     \item $m=4$ and $(H,S)=(C_2\times C_2,C_2)$, $(\Dih(4),C_2\times C_2)$, or $(\Dih(4),C_4)$;
        %     \item $m=5$, $H=\AGL_1(5)$ and $S=\Dih(5)$;                        
        %     \item $m=p+1$, $H=\PGL_2(p)$ and $S=\PSL_2(p)$, for $p\in\{5,7\}$.
        %  \end{itemize}
        \item $\qrk(\Alt(m),(1,2))$ for $m\ge 5$ (see Construction~\ref{con:nonEq});
        \item $H \wr C_2 \cong H^2 \rtimes C_2$ where $m$ and $H$ satisfy one of:
        \begin{itemize}
           \item $m\ge 3$ and $ H = \Sym(m)$;
           \item $m\ge 4$ and $ H = \Alt(m)$;
           \item $m=4$ and $H =C_4,\; C_2\times C_2$ or $\Dih(4)$;
           \item $m=5$ and $H=\AGL_1(5)$;
           \item $m=6$ and $H=\PGL_2(5)$;
           \item $m=8$ and $H=\PSL_2(7)$, $\PGL_2(7)$, $\AGL_1(8)$, $\AGammaL(1,8)$ or $\ASL_3(2)$;
           \item $m=9$ and $H=\PSL_2(8)$ or $\PGammaL_2(8)$;
           \item $m=12$ and $H=M_{12}$;
        \end{itemize}
       \item one of transitive groups of degree $6$ that
       can be found in the Library of Transitive Groups in \cite{magma} under the names {\tt TransitiveGroup}$(6,i)$ for $i \in \{9,10\}$. 
       \item one of transitive groups of degree $8$ that
       can be found in the Library of Transitive Groups in \cite{magma} under the names {\tt TransitiveGroup}$(8,i)$ for $i \in \{19,21,22,26,28,29,$ $30,31,33,34,41,45,46\}$.
       \item one of transitive groups of degree $10$ that
       can be found in the Library of Transitive Groups in \cite{magma} under the names {\tt TransitiveGroup}$(10,i)$ for $i \in \{17,19,20,27,41,$ $42\}$.
       \item one of transitive groups of degree $12$ that
       can be found in the Library of Transitive Groups in \cite{magma} under the names {\tt TransitiveGroup}$(12,i)$ for $i \in \{181,182,229,278,$ $279,297,298\}$.
       \item one of transitive groups of degree $14$ that
       can be found in the Library of Transitive Groups in \cite{magma} under the names {\tt TransitiveGroup}$(14,i)$ for $i \in \{59,60\}$.
       \item one of transitive groups of degree $16$ that
       can be found in the Library of Transitive Groups in \cite{magma} under the names {\tt TransitiveGroup}$(16,i)$ for $i \in \{1078, 1502, 1798,$ $1799, 1801, 1802, 1882, 1883, 1950, 1951\}$.
       \item one of transitive groups of degree $18$ that
       can be found in the Library of Transitive Groups in \cite{magma} under the names {\tt TransitiveGroup}$(18,i)$ for $i \in \{ 937, 938, 979,$ $980\}$.
    \end{enumerate}
\end{lemma}
\begin{remark}
    Observe that $\qrk(\Alt(m), (1,2)) \cong \mix(\Sym(m), \Alt(m), {\bf 1})$ and that  $\qrk(H,1) \cong H \times C_2$ for all $m \ge 3$. Besides these groups, the following groups appear both in \cref{lemma:caseA} and \cref{lemma:caseB}.
    \begin{itemize}
        \item For degree $4$, $C_2\times C_2$, $C_4$ or $\Dih(4)$, all acting transitively on $4$ points;
        \item For degree $8$,
        \begin{itemize}
            \item the groups {\tt TransitiveGroup}$(8,i)$ for $i \in \{21,22,26,28,29,30,31\}$, which appear in \cref{lemma:caseB} and are identified in \cref{lemma:caseA},
        \item the groups {\tt TransitiveGroup}$(8,17)$ and {\tt TransitiveGroup}$(8,18)$, which appear in \cref{lemma:caseA} and are identified in \cref{lemma:caseB}
        \item and the group {\tt TransitiveGroup}$(8,19)$, which is not indentified in either \cref{lemma:caseA} or \cref{lemma:caseB}.
        \end{itemize}
    \end{itemize}
\end{remark}

\begin{proof}
    As for \cref{lemma:caseA}, the statement has been verified computationally for $m\le 9$ (see \cite{ourAlgorithm}). Therefore, we can assume that $m\ge 10$.

    \smallskip
    \noindent \textsc{Let us show that all the permutation groups $G$ of degree $2m\ge 20$
    listed in \cref{lemma:caseB} satisfy
    $\m(G) = m+1$.} As every other group appearing is an overgroup of both $\Alt(m) \times C_2$ and $\qrk(\Alt(m), (1,2))$, we just need to consider these two cases. Observe that $\Alt(n) \times C_2$ has blocks of size two, and hence the conclusion follows from \cref{lemma:caseA}.
    
    Hence, we need to consider $\qrk(\Alt(m), (1,2))$. Let $A$ be a $k$-subset of points, for some $4 \le k \le n$. We can suppose that $A$ and $A^s$ are not disjoint. We label the points of $B$ by $\{1_B, 2_B, \dots, n_B\}$, and those of $C$ by $\{1_C, 2_C, \dots, n_C\}$. We can also assume that $1_B^s = 2_C$, $2_B^s = 1_C$, $1_C^s = 2_B$, $2_C^s = 1_B$, while, for every integer $3\le x \le n$, $x_B^s=x_C$ and $x_C^s = x_B$. We define the integers
    \[ z = |A \cap \left\lbrace 1_B,2_B,1_C,2_C \right\rbrace|, \quad z_B = |A \cap \left\lbrace 1_B,2_B \right\rbrace|, \quad z_C = |A \cap \left\lbrace 1_C,2_C \right\rbrace|.\]
    Observe that there exist two integers $a,b \le n$ such that
    \[ 4-z = |A \cap \left\lbrace a_B,b_B,a_C,b_C \right\rbrace|, \quad 2-z_B = |A \cap \left\lbrace a_B,b_B \right\rbrace|, \quad 2-z_C = |A \cap \left\lbrace a_C,b_C \right\rbrace|. \]
    It is immediate to check that there exists a double transposition $t$ such that $t \in \Alt(4)$ and
    \[ \left\lbrace 1_B,2_B,1_C,2_C,a_B,b_B,a_C,b_C \right\rbrace \cap \left\lbrace 1_B,2_B,1_C,2_C,a_B,b_B,a_C,b_C \right\rbrace^{st} = \emptyset .\]
    Finally, the pointwise stabilser of the set $\Omega - \left\lbrace 1_B,2_B,1_C,2_C,a_B,b_B,a_C,b_C \right\rbrace$ contains $\Alt(n-4)$. The choice of a suitable $h$ so that $A$ and $A^{sth}$ are disjoint is the same as in the corresponding case in \cref{lemma:caseA}.%(Note that our proof of the equality $\m(\Alt(n) \times C_2)=\m(\Sym(n)\times C_2)=n+1$ does not require $n\ge 8$, but rather the weaker $n\ge 5$: this is what forces us to take $n\ge 9$ in this proof.)

    \smallskip
    \noindent \textsc{Let us now prove that, if $\m(G) = m+1$ and $2m\ge 20$, then $G$ appears in the list in \cref{lemma:caseB}.}
    Suppose that $G$ admits a system of imprimitivity $\Sigma$ consisting of two blocks of size $m$. Let us denote these blocks by $B$ and $C$, and let $X, Y$ be two subsets of $B$ of equal size $k< m/2$.  We will show that there exists a permutation in $G$ mapping $X$ to $Y$.
    Let $g$ be an arbitrary permutation that swaps the two blocks $B$ and $C$. Consider the subset
    \[A = X \cup \left( C \setminus Y^{g^{-1}} \right) .\]
    Since $|A| = m < \m(G)$, there exists a permutation $h$ such that $A \cap A^h$ is empty. Observe that $X^h =Y^{g^{-1}}$, and hence $X^{hg} = Y$. By the arbitrariness of $X$ and $Y$, we conclude that the permutation group that $G^B_B$, induced by the action of $G_B$ on $B$, is $k$-homogeneous, for every $k < m/2$. Using again \cite[Section~7]{Cameron1981} and~\cite[Theorem~9.4B]{DixonMortimer} and that $m\ge 10$, we see that either $\Alt(n) \le G_B^B \le \Sym(n)$ or $M_{12} = G_B^B$. In the latter scenario, in view of \cref{lem:imprimitiveAS}, we have two candidate groups: namely, $M_{12} \wr C_2$ and $\qrk(M_{12},1) \cong M_{12} \times C_2$. (Note that \cref{con:nonEq} produces no other subgroup as ${\bf N}_{\Sym(12)}(M_{12}) = M_{12}$.) We have verified by a computer-aided calculation that $M_12 \wr C_2$ reaches the upper bound, while $M_{12} \times C_2$ does not. Therefore, we are left with $\Alt(n) \le G_B^B \le \Sym(n)$, and the conclusion follows by \cref{ex:cringe,lem:imprimitiveAS}. 
\end{proof}

\section{Primitive groups} \label{sec:primitive}
This section is devoted to the proofs of our results concerning primitive permutation groups. The subsections are organized according to the type of primitive group under consideration, as classified by the O'Nan–Scott Theorem. As previously mentioned, we follow the partition of primitive groups into eight classes proposed in \cite{Praeger1997}.

\subsection{Almost simple type} \label{sec:AS}
The proof of \cref{thm:AS} is divided into  several lemmas depending on the socle of the group considered.

We start by considering the action of the alternating group $\Alt(m)$ or the symmetric group $\Sym(m)$ on the set ${[m] \choose k}$ of $k$-subsets of the set $[m]=\{1,2,\ldots,m\}$. Observe that the degree of this permutation representation is ${m \choose k}$. (\cref{thm:ksets} takes care of the case~$(a)$ in \cref{table}.)
\begin{lemma}
\label{thm:ksets}
Let $2\le k < m$ be two integers. Then
$$\m\left( \Alt(m) \act {[m] \choose k} \right) \le \m\left( \Sym(m) \act {[m] \choose k} \right) \le {m-\lfloor \frac{k}{2} \rfloor \choose \lceil \frac{k}{2} \rceil}.$$
Consequently, if $\C_k$ is the family of symmetric (or alternating) groups acting on $k$-subsets, then
\[ \lim_n\frac{\F_{\C_k}\left(n\right)}{\sqrt{n}} \le \frac{\sqrt{k!}}{(k/2)!} \quad \hbox{ for } k \hbox{ even,}\]
and
\[ \lim_n\frac{\F_{\C_k}\left(n\right)}{n^{\frac{k+1}{2k}}} \le \frac{k!^{\frac{k+1}{2k}}}{\lceil k/2 \rceil!} \quad \hbox{ for } k \hbox{ odd} .\]
\end{lemma}

\begin{proof}
For the sake of simplicity, let $r = \lfloor k/2 \rfloor$.
Consider the subset
$$A = \left\{ X \in {[m]\choose k} \,\middle\vert\, \{1,2,\ldots,r\} \subseteq X\right\}.$$
For every $g\in \Sym(m)$, we compute
\[ A \cap A^g = \left\{ X \in {[m]\choose k} \,\middle\vert\, \{1,2,\ldots,r\} \cup \{1^g,2^g,\ldots,r^g\} \subseteq X\right\}.\]
This intersection is never empty, and, if $g$ does not stabilise $\{1,2,\ldots,r\}$, $A^g$ and $A$ are not equal. Hence, $A$ is not self-separable for $\Sym(m)$. This  proves the upper bound of $\m(\Sym(m))$.

Note that
\[ |A| = {m-r \choose k-r} = \frac{1}{(k-r)!}m^{k-r} + o \left(m^{k-r} \right), \quad \hbox{and} \quad n = {m \choose k } = \frac{1}{k!}m^{k} + o \left(m^{k} \right) \,.\]
A direct computation is now sufficient to determine the claimed asymptotic behaviours of $\f_\C(n)$ and $\F_\C(n)$.
\end{proof}

We now turn our attention to the standard actions of the classical groups. We are going to follow the notation established in \cite[Chapters~2 and~4]{BurnessGiudici2016}.

We start by considering the possible subspace actions of a classical group whose socle is isomorphic to $\mathrm{PSL}_d(q)$. There are three actions to consider: that on $k$-subspaces, that on pairs of complementing subspaces, and that on pairs of subspaces, one containing the other. From here on, the symbol $V^{(k)}$ denotes the set of $k$-dimensional subspaces of the natural module of a classical group.
(\cref{lem:PSL} takes care of the cases~$(b)$, $(c)$ and~$(d)$ in \cref{table}.)
\begin{lemma}\label{lem:PSL}
    Let $G$ be a classical group with linear socle and associated natural module $\FF_q^d$ endowed with a subspace action.
    \begin{enumerate}[$(i)$]
        \item If the domain of the action is the set of $k$-subsets, with $k\le d/2$, then
        \[\m(G) \le \left[ d-\lfloor k/2\rfloor \atop k-\lfloor k/2\rfloor\right]_q .\]
        Consequently, if $\C_k$ is the family of such permutation groups, then
        \[ \lim_n\frac{\f_{\C_k}\left(n\right)}{\sqrt{n}} = \lim_n\frac{\F_{\C_k}\left(n\right)}{\sqrt{n}} = 1 \quad \hbox{ for } k \hbox{ even,}\]
        and
        \[ \lim_n\frac{\F_{\C_k}\left(n\right)}{n^{\frac{k+1}{2k}}} \le 1 \quad \hbox{ for } k \hbox{ odd} .\] 
        \item If the domain of the action is the set of pairs of subspaces complementing each other, one of which of dimension $k\le d/2$, then
        \[\m(G) \le q^{k(d-k)}\left[ d-\lfloor k/2\rfloor \atop k-\lfloor k/2\rfloor\right]_q .\]
        Consequently, if $\C_k$ is the family of such permutation groups, then
        \[ \lim_n\frac{\f_{\C_k}\left(n\right)}{\sqrt{n}} = \lim_n\frac{\F_{\C_k}\left(n\right)}{\sqrt{n}} = 1 \quad \hbox{ for } k \hbox{ even,}\]
        and
        \[ \lim_n\frac{\F_{\C_k}\left(n\right)}{n^{\frac{k+1}{2k}}} \le 1 \quad \hbox{ for } k \hbox{ odd} .\]
        \item If the domain of the action is the set of pairs of subspaces, one of dimension $k < d/2$ and the other of dimension $d-k$, the smaller contained in the larger, then
        \[\m(G) \le \left[ \lceil (d-k)/2 \rceil \atop \lceil (d-3k)/2\rceil \right]_q \left[ \lceil(2d -3k) /2 \rceil \atop \lceil k/2\rceil \right]_q .\]
        Consequently, if $\C_k$ is the family of such permutation groups, then
        \[ \lim_n\frac{\F_{\C_k}\left(n\right)}{n^{\frac{k \left\lceil \frac{d - 3k}{2} \right\rceil + (d-k) \left\lceil \frac{k}{2} \right\rceil}{k(2d-3k)}}} \le 1 .\]
    \end{enumerate}
\end{lemma}
\begin{proof}
    We start by considering the action of $G$ on $k$-subspaces. Let
    \[U = \langle e_1,e_2,\ldots, e_{\lfloor k/2 \rfloor} \rangle ,\]
    and consider the subset
    $$A = \left\{ X \in V^{(k)} \,\middle\vert\,  U \leq X\right\}.$$
    For every $g\in G$, we can choose a subspace $X_g \in V^{(k)}$ such that
    \[ U + U^{g^{-1}} \leq X_g. \]
    (Note that such a subspace exists because $\dim (U + U^{g^{-1}}) \le k$.)
    By construction,
    \[ U^g + U \le X_g^g \in A \cap A^g .\]
    Hence $A\cap A^g$ is nonempty, and $A$ is self-separable.
    Observe that
    \[|A| = \left[ d-\lfloor k/2\rfloor \atop k-\lfloor k/2\rfloor\right]_q = q^{\left(k - \left\lceil\frac{k}{2}\right\rceil \right)(d-k)} + O\left( q^{\left(k - \left\lceil\frac{k}{2}\right\rceil \right)(d-k) -1}\right) ,\]
    and
    \[|V^{(k)}|=\left[ d \atop k \right]_q = q^{k(d-k)} + O\left( q^{k(d-k) -1}\right). \]
    Therefore, we can compute that
    \[ |A| \sim n^{\frac{1}{k} \left\lceil \frac{k}{2}\right\rceil} .\]
    This concludes the proof of~$(i)$.

    Let us now consider the action on pairs of complementary subspaces. Suppose that the smallest one has dimension $k$. Define
    \[ U = \langle e_1,e_2,\ldots, e_{\lfloor k/2 \rfloor} \rangle ,\]
    and consider
    $$A = \left\{ (X,Y) \in V^{(k)}\times V^{(d-k)} \,\middle\vert\,  V= X \oplus Y \hbox{ and } U \leq Y\right\}.$$
    For every $g\in G$, we choose $(X_g,Y_g) \in V^{(k)}\times V^{(d-k)}$ so that
    \[ U + U^{g^{-1}} \le X_g  \quad \hbox{and}  \quad V= X_g \oplus Y_g .\]
    A similar argument as in the first case shows that $A$ is not self-separable, because $(X_g,Y_g)^g \in A \cap A^g$. Moreover,
    \[|A| = q^{k(d-k)} \left[ d-\lfloor k/2\rfloor \atop k-\lfloor k/2\rfloor\right]_q = q^{k(d-k)\left(k - \left\lceil\frac{k}{2}\right\rceil \right)(d-k)} + O\left( q^{k(d-k)\left(k - \left\lceil\frac{k}{2}\right\rceil \right)(d-k) -1}\right) ,\]
    and
    \[|\Omega| = q^{k(d-k)}\left[ d \atop k \right]_q = q^{2k(d-k)} + O\left( q^{2k(d-k) -1}\right). \]
    Therefore,
    \[ |A| \sim n^{\frac{1}{k} \left\lceil \frac{k}{2}\right\rceil} ,\]
    which is enough to conclude~$(ii)$.

    Finally, we need to consider the domain of pairs of subspaces, say $X$ and $Y$, so that $X\in V^{(k)}$, $Y\in V^{(d-k)}$ and $X \le Y$. In this case, we have to set
    \[ U = \langle e_1, \ldots, e_{\lfloor k/2 \rfloor} \rangle\] and
    \[ W = \langle e_1, \ldots, e_{\lfloor (d-k)/2 \rfloor} \rangle.\]
    Considering the set
    \[A = \{ (X,Y) \in V^{(k)}\times V^{(d-k)} \mid X \leq Y,\, U \leq X,\, W \leq Y \},\]
    and using similar strategies as before, we conclude that $A$ is not self-separable.
    In particular, we have
    \begin{align*}
        |A| &= \left[ d- \lfloor (d-k)/2\rfloor \atop d - k -\lfloor (d-k)/2\rfloor \right]_q \left[ d - k - \lfloor k/2\rfloor \atop k -\lfloor k/2\rfloor \right]_q
        \\&= \left[ \lceil (d-k)/2 \rceil \atop \lceil (d-3k)/2\rceil \right]_q \left[ \lceil(2d -3k) /2 \rceil \atop \lceil k/2\rceil \right]_q
        \\&= q^{k \left\lceil \frac{d - 3k}{2} \right\rceil + (d-k) \left\lceil \frac{k}{2} \right\rceil} + O\left(q^{k \left\lceil \frac{d - 3k}{2} \right\rceil + (d-k) \left\lceil \frac{k}{2} \right\rceil - 1}\right)
    \end{align*}
    and
    \[|\Omega| = \left[ d-k \atop k \right]_q\left[ d \atop d-k \right]_q = q^{k(2d-3k)} + O\left( q^{k(2d-3k) -1}\right). \]
    Therefore,
    \[ |A| \sim n^{\frac{k \left\lceil \frac{d - 3k}{2} \right\rceil + (d-k) \left\lceil \frac{k}{2} \right\rceil}{k(2d-3k)}} . \]
    This completes the proof of the missing case~$(iii)$.
\end{proof}

Let now $G$ be a classical group whose socle is not isomorphic to $\mathrm{PSL}_d(q)$. We start by considering the action of $G$ on the totally isotropic or totally singular $k$-subspaces of its natural module. From here on, with the symbol $V_\bullet^{(k)}$ we refer to the set of $k$-dimensional subspaces of the formed space $(V,\mathbf{b})$ where $\bullet$ signals a given property, which is clear from context. For instance, in the next proof $V_\bullet^{(k)}$ denotes in a unified way the totally isotropic and totally singular $k$-subspaces. (\cref{lem:classicalTS} takes care of the cases~$(e)$, $(g)$, $(i)$, $(j)$ and~$(k)$ in \cref{table}. Moreover, we adopt the notation that if in a product $\prod_{i=a}^b x_i$, $a$ is greater than $b$, then the product is equal to $1$.) 

\begin{lemma}\label{lem:classicalTS}
    Let $G$ be a classical group whose socle is not linear, and consider its subspace action on totally isotropic or totally singular $k$-subspaces of its natural module $V$. Then there exists a non self-separable subset $A$ of the domain whose exact cardinality, as well as the cardinality of the domain, are collected in \cref{table:TS}.
\end{lemma}
Observe that in the first row of \cref{table:TS} the symbol $\delta$ appears. We have that $\delta = |(2,d)-2|$, that is, $\delta = 0$ if $d$ is even, and $\delta=1$ if $d$ is odd.

Moreover, the expressions for \( |\Omega| \) and \( |A| \) appearing in \cref{table:TS} are too complicated to yield a clean explicit asymptotic description. Instead, we content ourselves with the observation that our upper bound can be (trivially) rewritten as
\[
|\Omega|^{\log_{|\Omega|} |A|}.
\]
By numerical experimentation, for fixed \(k\), the exponent \(\log_{|\Omega|} |A|\) appears to approach \(1/2\) as \(d \to \infty\).

\begin{proof}[Proof of \cref{lem:classicalTS}]
    For the sake of uniformity, let us introduce the parameter
    \[m = \dim(V) - 2 \dim(M) ,\]
    where $M$ is any totally isotropic or or totally singular subspace of maximal dimension.  We fix a totally isotropic or a totally singular subspace $U$ with
    \[ \dim(U) = \frac{\dim(V)-m}{2} - k.\]
    Consider the set
    \[ A = \left\{ X \in V^{(k)}_{\bullet} \mid X \perp U \right\} ,\]
    and let $M$ be a totally isotropic or totally singular subspace of maximal dimension containing $U$.
    For every $g \in G$,  we define
    \[ X_g = (U+U^{g^{-1}})^\perp \cap M. \]
    Note that
    \[\dim(X_g) = \dim(M) - \dim\left(  \frac{U^{g^{-1}}}{U^{g^{-1}} \cap M} \right) \ge \frac{\dim(V)-m}{2} - \dim(U) =  k  .\]
    We can now choose $Y^g$, a $k$-subspace of $X_g$. Observe that
    \[ Y_g \le X_g \le  M = M^\perp \le X_g^\perp \le Y_g^\perp \]
    and that, by construction,
    \[ Y_g^g \perp U^g + U ,\]
    and thus $Y_g^g \in A \cap A^g$. Therefore, $A$ is not self-separable.
    
    We now need to focus to compute $|A|$ and the degree of the actions. Note that the latter is the number of totally isotropic or total singular $k$-subspaces in a given formed space: this value can be found in \cite[Table~4.1.2]{BurnessGiudici2016}. On the other hand, to count the elements of $A$ we just need to count the number of totally isotropic or totally singular subspaces of $U^\perp$. Observe that, as $\mathbf{b}\downarrow_U$ is trivial, $U^\perp$ can be written as the orthogonal direct sum $U^\perp = U \oplus U^\perp/U$, where $U$ is totally isotropic or totally singular, and $(U^\perp/U, \mathbf{b}\downarrow_{U^\perp/U})$ is a formed space of the same type as $(V, \mathbf{b})$. Furthermore,
    \[\dim(U^\perp) = \dim(V) - \dim(M) + k = \frac{\dim(V)+m}{2} + k, \]
    and
    \[\dim(U^\perp/U) = \dim(V) - 2 \dim(M) + 2k = 2k + m .\]
    To every $X \in A$, we can uniquely assign a direct sum decomposition
    \[ X = \left( X \cap U \right) \oplus \left( X \cap U^\perp/U \right) .\]
    Hence, $H$ can be identified with a pair $(X_0,X_1)$, where $X_1$ is a totally isotropic or totally singular $h$-subspace in $U^\perp/U$, and $X_0$ with a $(k-h)$-subspace of $U$. Therefore, if $N_h(U^\perp/U)$ denotes the number of totally isotropic or totally singular $h$-subspaces of $U^\perp/U$, we proved that
    \[ |A| = \sum_{h=\max\{0, 2k-\dim(M)\}}^{k} \left[ \dim(M)-k \atop k-h \right]_q N_h(U^\perp/U) .\]
    (Note that the Gaussian coefficient loses meaning if $\dim(M)-k < k-h$. This explains why $h$ varies in the prescribed range.) 
    Finally, by using the formulae of \cite[Table~4.1.2]{BurnessGiudici2016}, we can compute $|A|$. We have explained a way to compute all the data appearing in \cref{table:TS}, which completes the proof of \cref{lem:classicalTS}. \hfill \qedhere

{\small
    \begin{sidewaystable}
	\centering
	\rowcolors{2}{white}{OliveGreen!25}
	\centering
	\begin{tabularx}{\textwidth}{l l X X}
		\toprule
        & $\mathrm{soc}(G)$ & $|\Omega|$ & $|A|$ %& Asymptotics 
        \\
		
		\midrule
        $(i)$ & $\mathrm{PSU}_d(q)$ & $\displaystyle \frac{\prod\limits_{i = d-2k+1}^{d}(q^i - (-1)^i)}{\prod\limits_{i = 1}^k (q^{2i}-1)}$ & $\displaystyle \sum\limits_{h=\max\{0, 2k - \lfloor \frac{d}{2} \rfloor\}}^{k} \left[ \lfloor \frac{d}{2} \rfloor - k \atop k-h \right]_q
        \frac{\prod\limits_{i = 2k - 2h + 1 + \delta}^{2k + \delta }(q^i - (-1)^i)}{\prod\limits_{i = 1}^h (q^{2i}-1)}$
        %& \begin{cases}
           % \sim n ^{\frac{d - 2k}{(2d-3k)k}} &\hbox{if } 0 \ge 2k - \lfloor \frac{d}{2}\rfloor \\
           % \sim n ^{\frac{\left(-2k + d + \lceil \frac{d}{2} \rceil \right ) \left (2k - \lfloor \frac{d}{2} \rfloor\right )}{(2d-3k)k}} &\hbox{if } 0 < 2k - \lfloor \frac{d}{2}\rfloor
        %\end{cases} 
        \\
        
        $(ii)$ & $\mathrm{PSp}_{2d}(q)$ & $\displaystyle \frac{\prod\limits_{i = d-k+1}^{d}(q^{2i} - 1)}{\prod\limits_{i = 1}^k (q^{i}-1)}$& 
        $\displaystyle \sum\limits_{h=\max\{0, 2k - d\}}^{k} \left[ d - k \atop k-h \right]_q 
        \frac{\prod\limits_{i =k - h +1}^{k}(q^{2i} - 1)}{\prod\limits_{i = 1}^h (q^{i}-1)}$% & \begin{cases}
            %\sim n ^{\frac{2d - 4k}{(4d-3k + 1)k}} &\hbox{if } 0 \ge 2k - d \\
            %\sim n ^{\frac{\left(3d - 2k + 1 \right ) \left (2k - d\right )}{(2d-3k)k}} &\hbox{if } 0 < 2k - d
        %\end{cases} 
        
        \\
        
        $(iii)$ & $\mathrm{P}\Omega_{2d}^+(q)$ & $\displaystyle \frac{(q^d-1)(q^{d-k}+1)\prod\limits_{i = d-k+1}^{d-1}(q^{2i}-1)}{2^{\delta (d,k)}\prod\limits_{i = 1}^k (q^{i}-1)}$ & $\displaystyle \sum\limits_{h=\max\{0, 2k - d\}}^{k} \left[ d - k \atop k-h \right]_q
        \frac{(q^k-1)(q^{k-h}+1)\prod\limits_{i = k-h+1}^{k-1}(q^{2i}-1)}{2^{\delta (k,h)}\prod\limits_{i = 1}^h (q^{i}-1)}$
        %& $\displaystyle \sim n^{\frac{1}{k} \left\lceil \frac{k}{2}\right\rceil}$ 
        \\
        
        $(iv)$ & $\mathrm{P}\Omega_{2d}^-(q)$ &  $\displaystyle \frac{(q^d+1)(q^{d-k}-1)\prod\limits_{i = d-k+1}^{d-1}(q^{2i}-1)}{\prod\limits_{i = 1}^k (q^{i}-1)}$ & $\displaystyle \sum\limits_{h=\max\{0, 2k - d -1\}}^{k} \left[ d -1 - k \atop k-h \right]_q
        \frac{(q^{k+1}+1)(q^{k-h+1}-1)\prod\limits_{i = k-h+2}^{k}(q^{2i}-1)}{\prod\limits_{i = 1}^h (q^{i}-1)}$
        %&
        \\
        
        $(v)$ & $\Omega_{2d+1}(q)$ & $\displaystyle \frac{\prod\limits_{i = d-k+1}^{d}(q^{2i}-1)}{\prod\limits_{i = 1}^k (q^{i}-1)}$ & $\displaystyle \sum\limits_{h=\max\{0, 2k - d \rfloor\}}^{k} \left[ d - k \atop k-h \right]_q
        \frac{\prod\limits_{i = k-h+1}^{k}(q^{2i}-1)}{\prod\limits_{i = 1}^h (q^{i}-1)}$ %& $\displaystyle \sim n^{\frac{1}{k} \left\lceil \frac{k}{2}\right\rceil}$ 
        \\
        \bottomrule
	\end{tabularx}%
	{\medskip \caption{Computation of degree, $|A|$, and asymptotics for $|A|$ for classical groups acting on totally singular or totally isotropic $k$-subspaces.}%
	\label{table:TS}}%
\end{sidewaystable}%
}
\end{proof}

Once again, let $G$ be a classical group whose socle is not isomorphic to $\mathrm{PSL}_d(q)$. We study the action of $G$ on the nondegenerate $k$-subspaces of its natural module. (\cref{lem:classicalND} takes care of the cases~$(f)$, $(h)$, $(m)$, $(n)$, $(o)$, $(p)$, $(q)$, $(r)$ and~$(s)$ in \cref{table}.) 

\begin{lemma}\label{lem:classicalND}
    Let $G$ be a classical group whose socle is not linear, and consider its subspace action on nondegenerate $k$-subspaces of its natural module $V$. Then there exists a non self-separable subset $A$ of the domain whose exact cardinality and asymptotic size in relation with the degree are collected in \cref{table:ND}.
\end{lemma}
In \cref{table:ND}, the letters $B$ and $C$ refer to the functions
\[ B (k,d) = -2  + 3d -2k \quad \hbox{and} \quad C(k,d) = -2 + 4d -4k. \]
In particular, it is interesting for note that, for $k$ fixed,
\[ \lim_d \frac{B(k,d)}{C(k,d)} = \frac{3}{4} \quad \hbox{and} \quad \lim_d \frac{1}{C(k,d)}=0 . \]
\begin{proof}
    As in the proof of \cref{lem:classicalTS}, we start by fixing a subspace $U$. Unlike before, we need to make different choices of $U$ depending on the socle of $G$ and, if the bilinear form is symmetric, depending whether the subspaces we are acting on are hyperbolic, parabolic or elliptic. Our choices for $U$ are as follows.
    \begin{enumerate}[$(i)$]
        \item If $\mathrm{soc}(G)$ is unitary, and we are acting on nondegenerate $k$-subspaces, $U$ is any nondegenerate $\lfloor k/2 \rfloor$-subspace.
        \item If $\mathrm{soc}(G)$ is symplectic, and we are acting on nondegenerate $2k$-subspaces, $U$ is any nondegenerate $2\lfloor k/2 \rfloor$-subspace.
        \item If $\mathrm{soc}(G)$ is orthogonal, and we are acting on nondegenerate hyperbolic $2k$-subspaces, $U$ is any nondegenerate hyperbolic $2\lfloor k/2 \rfloor$-subspace.
        \item If $\mathrm{soc}(G)$ is orthogonal, and we are acting on nondegenerate parabolic $(2k+1)$-subspaces, $U$ is any nondegenerate hyperbolic $2\lfloor k/2 \rfloor$-subspace.
        \item If $\mathrm{soc}(G)$ is orthogonal, and we are acting on nondegenerate elliptic $2k$-subspaces, $U$ is any nondegenerate hyperbolic $2\lfloor (k-1)/2 \rfloor$-subspace.
    \end{enumerate}
    Consider
    \[ A = \left\{ X \in V^{(h)}_{\bullet} \mid U \le X\right\} ,\]
    where $h \in \{k, 2k, 2k+1\}$ depending on the case we are considering. For every $g\in G$, define
    \[X_g = U + U^{g^{-1}}.\]
    Note that the restricted form $\mathbf{b} \downarrow _{X_g}$ is sequilinear, alternating or symmetric (hyperbolic or parabolic depending on the dimension of $X$) according to the socle being unitary, symplectic or orthogonal, respectively. Furthermore,
    \[ \dim(X_g) \le 2\dim(U) \le 2^\epsilon k - \delta ,\]
    where $\epsilon = 0$ if the socle is unitary, and $\epsilon=1$ otherwise, and $\delta=1$ if the socle is orthogonal and the formed subspace is parabolic, $\delta=2$ if the socle is orthogonal and the formed subspace is elliptic, and $\delta=0$ otherwise. It is now immediate to build a subspace $Y_g \in V^{(k)}_\bullet$ such that $X_g \le Y_g$. Moreover, by construction, $Y_g \in A \cap A^g$. Therefore, $A$ is not self-separable.

    We need to deal with computing the size of $A$. Observe that, in any case, $(U, \mathbf{b}\downarrow_U)$ is a nondegenerate formed space whose bilinear form is either sesquilinear, alternating or hyperbolic bilinear, depending on the socle of $G$. Therefore, by quotienting by $U$, $A$ can be thought of as the set of nondegenerate $(h-\dim(U))$-subspaces of $V/U$. Therefore, the formulae for computing $|A|$ and the degree are readily available in \cite[Table~4.1.2]{BurnessGiudici2016}. This information completes the proof of \cref{lem:classicalND}. \hfill $\qedhere$

{\small
    \begin{sidewaystable}
	\centering
	\rowcolors{2}{white}{OliveGreen!25}
	\centering
	\begin{tabularx}{\textwidth}{l l l X X l}
		\toprule
        & $\mathrm{soc}(G)$ & Domain & $|\Omega|$ &  $|A|$ & Asymptotics \\
		
		\midrule
        $(i)$ & $\mathrm{PSU}_d(q)$ & $k$-nondegenerate & $\displaystyle \frac{q^{k(d-k)} \prod \limits_{i=d-k+1}^d(q^i - (-1)^i)}{\prod \limits_{i = 1}^{k}(q^i - (-1)^i)}$ & $\displaystyle \frac{q^{\lceil \frac{k}{2}\rceil (d-k)} \prod \limits_{i=d-k+1}^{d - \lfloor \frac{k}{2}\rfloor}(q^i - (-1)^i)}{\prod \limits_{i = 1}^{\lceil\frac{k}{2}\rceil}(q^i - (-1)^i)}$ & $\displaystyle \lesssim n^{\frac{1}{k} \left\lceil \frac{k}{2}\right\rceil}$ \\
        
        $(ii)$ & $\mathrm{PSp}_{2d}(q)$ & $2k$-nondegenerate & $\displaystyle \frac{q^{2k(d-k)} \prod \limits_{i=d-k+1}^d(q^{2i} - 1)}{\prod \limits_{i = 1}^{k}(q^{2i} - 1)}$ & $\displaystyle \frac{q^{2\lceil \frac{k}{2}\rceil (d-k)} \prod \limits_{i=d-k+1}^{d - \lfloor \frac{k}{2}\rfloor}(q^{2i} - 1)}{\prod \limits_{i = 1}^{\lceil\frac{k}{2}\rceil}(q^{2i} - 1)}$ & $\displaystyle \lesssim n^{\frac{1}{k} \left\lceil \frac{k}{2}\right\rceil}$ \\
        
        $(iii)$ & $\mathrm{P}\Omega_{2d}^+(q)$ & $2k$-hyperbolic & $\displaystyle \frac{q^{2k(d-k)}(q^d - 1)\prod \limits_{i=d-k}^{d-1}(q^{2i} - 1)}{2(q^k -1)(q^{d-k}-1)\prod \limits_{i = 1}^{k-1}(q^{2i} - 1)}$ & $\displaystyle \frac{q^{2\lceil \frac{k}{2}\rceil (d-k)}(q^{d - \lfloor \frac{k}{2}\rfloor} - 1)\prod \limits_{i=d-k}^{d - \lfloor \frac{k}{2}\rfloor -1}(q^{2i} - 1)}{2(q^{\lceil\frac{k}{2}\rceil} -1)(q^{d-k}-1)\prod \limits_{i = 1}^{\lceil\frac{k}{2}\rceil-1}(q^{2i} - 1)}$ & $\displaystyle \lesssim n^{\frac{1}{k} \left\lceil \frac{k}{2}\right\rceil}$\\
        
        $(iv)$ & $\mathrm{P}\Omega_{2d}^+(q)$ &  $(2k+1)$-parabolic & $\displaystyle \frac{q^{(2k+1)(2d -2k -1) -1}(q^d - 1)\prod \limits_{i=d-k}^{d-1}(q^{2i} - 1)}{2\prod \limits_{i = 1}^{k}(q^{2i} - 1)}$ & $\displaystyle \frac{q^{(2\lceil\frac{k}{2}\rceil+1)(2d -2k -1) -1}(q^{d - \lfloor \frac{k}{2}\rfloor} - 1)\prod \limits_{i=d-k}^{d - \lfloor \frac{k}{2}\rfloor -1}(q^{2i} - 1)}{2\prod \limits_{i = 1}^{\lceil\frac{k}{2}\rceil}(q^{2i} - 1)}$ & $\displaystyle \lesssim n^{\frac{\lceil \frac{k}{2}\rceil + \frac{B}{C} - \frac{1}{C}\lfloor \frac{k}{2}\rfloor}{k + \frac{B}{C}}}$ \\
        
        $(v)$ & $\mathrm{P}\Omega_{2d}^+(q)$ & $2k$-elliptic & $\displaystyle \frac{q^{2k(d-k)}(q^d - 1)\prod \limits_{i=d-k}^{d-1}(q^{2i} - 1)}{2(q^k + 1)(q^{d-k}+1)\prod \limits_{i = 1}^{k-1}(q^{2i} - 1)}$ &
        $\displaystyle \frac{q^{2\lceil \frac{k+1}{2}\rceil(d-k)}(q^{d - \lfloor \frac{k-1}{2}\rfloor} - 1)\prod \limits_{i=d-k}^{d- \lfloor \frac{k-1}{2}\rfloor -1}(q^{2i} - 1)}{2(q^{\lceil \frac{k+1}{2} \rceil} + 1)(q^{d-k}+1)\prod \limits_{i = 1}^{\lceil \frac{k-1}{2} \rceil}(q^{2i} - 1)}$ & $\displaystyle \lesssim n^{\frac{1}{k} \left\lceil \frac{k+1}{2} \right\rceil}$ \\
        
        $(vi)$ & $\mathrm{P}\Omega_{2d}^-(q)$ &  $(2k+1)$-parabolic & $\displaystyle \frac{q^{(2k+1)(2d -2k -1) -1}(q^d + 1)\prod \limits_{i=d-k}^{d-1}(q^{2i} - 1)}{2\prod \limits_{i = 1}^{k}(q^{2i} - 1)}$ &
        $\displaystyle \frac{q^{(2\lceil \frac{k}{2}\rceil+1)(2d -2k -1) -1}(q^{d - \lfloor\frac{k}{2} \rfloor} + 1)\prod \limits_{i=d-k}^{d- \lfloor\frac{k}{2} \rfloor-1}(q^{2i} - 1)}{2\prod \limits_{i = 1}^{\lceil \frac{k}{2} \rceil}(q^{2i} - 1)}$ & $\displaystyle \lesssim n^{\frac{\lceil \frac{k}{2}\rceil + \frac{B}{C} - \frac{1}{C}\lfloor \frac{k}{2}\rfloor}{k + \frac{B}{C}}}$\\
        
        $(vii)$ & $\mathrm{P}\Omega_{2d}^-(q)$ &  $2k$-elliptic & $\displaystyle \frac{q^{2k(d -k)}(q^d + 1)\prod \limits_{i=d-k}^{d-1}(q^{2i} - 1)}{2(q^k + 1)(q^{d-k}-1)\prod \limits_{i = 1}^{k-1}(q^{2i} - 1)}$ & $\displaystyle \frac{q^{2\lceil\frac{k+1}{2}\rceil (d -k)}(q^{d-\lfloor \frac{k-1}{2} \rfloor} + 1)\prod \limits_{i=d-k}^{d \lfloor \frac{k-1}{2} \rfloor -1}(q^{2i} - 1)}{2(q^{\lceil \frac{k+1}{2} \rceil} + 1)(q^{d-k}-1)\prod \limits_{i = 1}^{\lceil \frac{k-1}{2} \rceil}(q^{2i} - 1)}$ & $\displaystyle \lesssim n^{\frac{1}{k} \left\lceil \frac{k+1}{2} \right\rceil}$ \\
        
        $(viii)$ & $\Omega_{2d+1}(q)$ &   $2k$-hyperbolic & $\displaystyle \frac{q^{k(2d-2k+1)}\prod \limits_{i=d-k+1}^{d}(q^{2i} - 1)}{2 (q^k - 1)\prod \limits_{i = 1}^{k-1}(q^{2i} - 1)}$ & $\displaystyle \frac{q^{\lceil \frac{k}{2}\rceil (2d-2k+1)}\prod \limits_{i=d-k+1}^{d - \lfloor \frac{k}{2}\rfloor}(q^{2i} - 1)}{2 (q^{\lceil \frac{k}{2}\rceil} - 1)\prod \limits_{i = 1}^{\lceil \frac{k}{2} \rceil -1}(q^{2i} - 1)}$ & $\displaystyle \lesssim n^{\frac{1}{k} \left\lceil \frac{k}{2}\right\rceil}$\\
        
        $(ix)$ & $\Omega_{2d+1}(q)$ &   $2k$-elliptic & $\displaystyle \frac{q^{k(2d+1-2k)}\prod \limits_{i=d-k+1}^{d}(q^{2i} - 1)}{2 (q^k + 1)\prod \limits_{i = 1}^{k-1}(q^{2i} - 1)}$ & $\displaystyle \frac{q^{\lceil \frac{k+1}{2}\rceil (2d-2k+1)}\prod \limits_{i=d-k+1}^{d-\lfloor \frac{k-1}{2}\rfloor}(q^{2i} - 1)}{2 (q^{\lceil \frac{k+1}{2}\rceil} + 1)\prod \limits_{i = 1}^{\lceil \frac{k-1}{2}\rceil}(q^{2i} - 1)}$ & $\displaystyle \lesssim n^{\frac{1}{k} \left\lceil \frac{k+1}{2} \right\rceil}$ \\
        \bottomrule
	\end{tabularx}%
	{\medskip \caption{Computation of degree, $|A|$, and asymptotics for $|A|$ for classical groups acting on nondegenerate $h$-subspaces.}%
    \label{table:ND}}%
\end{sidewaystable}%
}
\end{proof}

We end \cref{sec:AS} on a sour note. An attentive reader might have noticed that we have not discussed case~$(l)$ in \cref{table}, which is the last step to complete the proof of \cref{thm:AS}. In this case, the socle of $G$ is isomorphic to $\Omega^\epsilon_{2d}(2^f)$, and the domain consists of nonsingular $1$-subspaces of the natural module of $G$. As stated in \cref{rem:AS}, our current method does not give us any control on $1$-dimensional subspaces, and we have not developed an \emph{ad hoc} approach for this scenario. Hence, the bound that appears in \cref{thm:AS} is actually the general bound from \cref{thm:upperBound}.

\subsection{Simple diagonal type} \label{sec:SD}
Let us first recall the definition of primitive groups of \emph{simple diagonal type}. Let $T$ be a nonabelian simple group, and let $k \geq 3$ be an integer. Note that we can naturally define three actions on $T^k$ -- the right-regular permutation representation of $T^k$, the component-wise action of $\Aut(T)$, and the action of $\Sym(k)$ permuting indices. We can consider the action induces by these actions on the domain $\Omega = T^k / \mathrm{diag}(T^k)$, where
\[\mathrm{diag}(T^k) = \{(t, \ldots, t) \mid t \in T\} \,,\]
Observe that the latter two actions preserve the diagonal of $T^k$, hence they fall in the stabilser of a point. For simplicity, we can identify $\Omega$ with $T^{k-1}$ via the set of representatives
\[\{(t_1, \ldots, t_{k-1}, 1) \mid t_1, \ldots, t_{k-1} \in T\}   \,.\]
Observe that, via this identification, the first $k-1$ components of $T^k$ act on $ \Omega$ regularly by right multiplication, while the last component acts semiregularly by left multiplication. In particular, every inner automorphism can be realized by choosing two suitable elements in $T^{k-1}$ and $T$. Therefore, every subgroup of $\Sym(\Omega)$ that preserves the algebraic structure of the domain as $T^k / \mathrm{diag}(T^k)$ can be written as
\[ T^k \rtimes \left( \Out(T) \times \Sym(k) \right) \]
endowed with the action we just described.
Therefore, a primitive permutation group $G$ is of simple diagonal type if
\[T^k \trianglelefteq G \leq T^k \rtimes \left( \Out(T) \times \Sym(k) \right) \,.\]
Note that, to ensure that $G$ is primitive, we need to impose that the action it induces on the $k$ direct factors of $T^k$ is primitive. Meanwhile $G$ is quasiprimitive, if this action is transitive. 

\begin{proof}[Proof of \cref{thm:SD}]
    In view of \cref{lemma:subgroup}, we can concentrate on $G = T^k \rtimes (\Out(T) \times S_k)$, the largest primitive group of diagonal type with socle $T^k$ acting on $T^{k-1}$. We will build a set $A \subseteq T^{k-1}$ that is not self-separable.

    First, let $B$ be any set of elements of $T$ such that
    \[ |B| = \frac{4}{\sqrt{3}}  |T|^{\frac{1}{2}} ,\]
    $1 \in B$, and for every $t\in T$, $B\cap Bt$ is nonempty. (Recall that its existence is guaranteed by \cref{thm:regular}.) Define \[A_0 = B \times \ldots \times B\] as the direct product of $k-1$ copies of $B$. Note that
    \[ |A_0| = |B|^{k-1} =  \left(\frac{16}{3} |T| \right)^{\frac{k-1}{2}} .\]
    Next, let \[A_1 = T A_0 ,\]
    where the left multiplication by $T$ corresponds to the semiregular action of the last component of the socle.
    Observe that
    \[ |A_1| \leq |A_0||T| = \left(\frac{16}{3}\right)^{\frac{k-1}{2}} |T|^{\frac{k+1}{2}}  .\]
    Finally, we let
    \[A = A_1^{\Out(T)}.\]
    If $T$ is not of Lie type, then $|\Out(T)|\le 4$. Meanwhile, if $T$ is of Lie type, a direct computation shows that $|\Out(T)| \leq 
    \log_2|T|$ (see, for example, \cite{Kohl2003}). Therefore, by setting $\epsilon = 1$ if $T$ is of Lie type, and $\epsilon = 0$ otherwise,
    \begin{align*}
        |A| &\leq |\Out(T)| |A_1| \\
            &\leq 4 \left(\frac{16}{3}\right)^{\frac{k-1}{2}} |T|^{\frac{k+1}{2}} \left(\frac{1}{4}\log_2|T|\right)^\epsilon\\
            & = 4 \left(\frac{16}{3}\right)^{\frac{k-1}{2}} n^{\frac{1}{2} + \frac{1}{k-1}} \left(\frac{1}{4k}\log_2 n\right)^\epsilon. 
    \end{align*}
    
    It remains to prove that the set $A$ satisfies our conditions, that is, for an arbitrary $g \in G$ we have that $A \cap A^g$ is nonempty. For every $g\in G$, there exist $a \in \Sym(k)$, $b\in \Out(t)$, $c\in T$ and $d \in T^{k-1}$ such that $g = abcd$. We claim that, by construction, $A$ is stabilized by $\Sym(k)$, by $\Out(T)$ and by the left semiregular action of $T$. The claim is clear for the latter two actions, while we shall prove it for $\Sym(k)$. Recall that $\Sym(k)$ acts on $T^k$ by permuting components, and hence, for every point $(t_1, \ldots, t_{k-1},1)\mathrm{diag}(T^k)$ and for every $a\in \Sym(k)$,
    \begin{align*}
        \mathrm{diag}(T^k) (t_1, t_2, \dots, t_{k-1},1)  ^a
        &= \mathrm{diag}(T^k) (t_{1a^{-1}}, t_{2a^{-1}}, \dots, t_{(k-1)x^{-1}},t_{ka^{-1}})  \\
        &= \mathrm{diag}(T^k) (t_{ka^{-1}}^{-1}t_{1a^{-1}}, t_{ka^{-1}}^{-1}t_{2a^{-1}}, \dots, t_{ka^{-1}}^{-1}t_{(k-1)a^{-1}},1) .
    \end{align*}
    (For simplicity, we denote the action of $\Sym(k)$ on the indices as right multiplication. Moreover, to avoid a cumbersome notation, we identify $1=t_k$.)
     Observe that every direct factor of $A$ is equal, and hence the previous formula proves the claim, because
    \[A ^{\Sym(k)} \subseteq TA = A .\]
    Hence, $A^{abc}=A$. Finally, we act with $d$ by component-wise right multiplication. By our choice of $B$, it follows that
    \[ |A \cap A^d | \ge |A_0 \cap A_0^d | \ge 1 .\]
    Therefore, $A$ is not self-separating for $G$, which completes the proof of the upper bound for $\m(G)$.

    The asymptotic estimates for $\F_\C(n)$ are obtained by restricting the attention to the groups of simple diagonal type whose socle is isomorphic to a $k$-fold Cartesian product of groups of Lie type. (Actually, groups of Lie type of characteristic $2$ are those for which the bound on $|\Out(T)|$ cannot be improved, see \cite{Kohl2003}, and thus meet the maximum in $\F_\C(n)$.)
\end{proof}

We observe that it would be tempting to retrace the same proof we used for simple diagonal for every family of primitive permutation groups whose socle is regular and nonabelian. Unfortunately, this approach fails because these groups have automorphisms groups which contain an isomorphic image of the socle, and hence the resulting upper bound would exceed the degree.

Moreover, in the proof of \cref{thm:SD}, our starting choice of $A_0$ implies that the size of $A$ depends on $k$ as an exponential function. Meanwhile, if we would have chosen $A_0$ to be any difference basis of the socle of appropriate size, by imposing the stability of $A$ under the action of $\Sym(k)$, the constant would have been factorial in $k$. Perhaps, a similar trick could extend our approach beyond the permutation groups of simple diagonal type -- but there is no reason not to suspect that the trick would be specific to some limited family of socles, rather than an arbitrary semisimple group.

\bibliographystyle{plain}
\bibliography{refs.bib}

\end{document}